\documentclass[12pt]{article}
\usepackage{amsmath,amssymb,latexsym}
\usepackage{ifpdf}
\ifpdf
\usepackage[pdftex]{graphicx}
\else
\usepackage[dvips]{graphicx}
\fi

\usepackage{epstopdf}

\newtheorem{theorem}{Theorem}[section]
\newtheorem{lemma}[theorem]{Lemma}
\newtheorem{corollary}[theorem]{Corollary}
\newtheorem{proposition}[theorem]{Proposition}

\newtheorem{example}[theorem]{Example}

\newenvironment{proof}
{\par\addvspace{0.3cm}\noindent{\rm Proof. }}
{\nopagebreak\mbox{}\hfill $\Box$\par\addvspace{0.25cm}}

\renewcommand{\Re}{\mbox{\rm Re\,}}
\renewcommand{\Im}{\mbox{\rm Im\,}}
\newcommand{\diag}{\mbox{\rm diag\,}}
\newcommand{\im}{\mbox{\rm im\,}}

\newcommand{\wind}{\mbox{\rm wind\,}}
\newcommand{\ind}{\mbox{\rm ind\,}}

\newcommand{\R}{{\mathbb R}}
\newcommand{\Rp}{{\mathbb R}_+}
\newcommand{\C}{{\mathbb C}}
\newcommand{\Z}{{\mathbb Z}}
\newcommand{\T}{{\mathbb T}}
\renewcommand{\kappa}{\varkappa}
\newcommand{\qed}{\hfill $\Box$}
\newcommand{\be}{\begin{equation}}
\newcommand{\ee}{\end{equation}}
\newcommand{\ds}{\displaystyle}

\newcommand{\bq}{\begin{eqnarray}}
\newcommand{\eq}{\end{eqnarray}}
\newcommand{\nn}{\nonumber}
\newcommand{\ba}{\begin{array}}
\newcommand{\ea}{\end{array}}
\newcommand{\bt}{\bar{\tau}}

\newcommand{\iv}{^{-1}}
\newcommand{\iy}{\infty}

\newcommand{\ovl}{\overline}
\newcommand{\Hp}{H^p(\T)}
\newcommand{\Hq}{H^q(\T)}
\newcommand{\Lp}{L^p(\T)}

\newcommand{\cL}{{\cal L}}
\newcommand{\eps}{{\varepsilon}}

\newcommand{\wh}{\widehat}

\newcommand{\cM}{\mathcal{M}}
\newcommand{\cS}{\mathcal{S}}

\newcommand{\cA}{\mathcal{A}}
\newcommand{\cB}{\mathcal{B}}

\newcommand{\cP}{\mathcal{P}}
\newcommand{\ta}{\tilde{a}}
\newcommand{\tb}{\tilde{b}}
\newcommand{\tc}{\tilde{c}}
\newcommand{\td}{\tilde{d}}
\newcommand{\xR}{\overline{\R}}
\newcommand{\tf}{\tilde{f}}

\begin{document}

\date{}
\title{Fredholm and invertibility theory for a special class of Toeplitz + Hankel operators}
\author{Estelle L. Basor\thanks{ebasor@aimath.org}\\
               American Institute of Mathematics\\
              Palo Alto, CA  94306, USA
        \and
        Torsten Ehrhardt\thanks{tehrhard@ucsc.edu, supported in part by NSF Grant DMS-0901434}\\
        Department of Mathematics\\
        University of California\\
       Santa Cruz, CA 95064, USA}
\maketitle

\begin{abstract}
We develop a complete Fredholm and invertibility theory for
Toeplitz+Hankel operators $T(a)+H(b)$ on the Hardy space $\Hp$, $1<p<\iy$, with piecewise continuous functions
$a,b$ defined on the unit circle which are subject to the condition $a(t)a(t\iv)=b(t)b(t\iv)$, $|t|=1$.
In particular, in the case of Fredholmness, formulas for the defect numbers are established. The results are applied to
several important examples. 

\end{abstract}

%%%%%%%%%%%%%%%%%%%%%%%%%%%%%%%%%%%%%%%%%%%%%%%%%%%%%%%%%%%%%%%%%%
%%%%%%%%%%%%%%%%%%%%%%%%%%%%%%%%%%%%%%%%%%%%%%%%%%%%%%%%%%%%%%%%%%

\section{Introduction}
\label{s1}

Let $\T=\{\,z\in\C\,:\,|z|=1\,\}$ stand for the unit circle in the complex plane, and let 
$L^p(\T)$ stand for the usual space of Lebesgue measurable $p$-integrable functions defined on $\T$,
where $1\le p\le \iy$. For a function $f\in L^1(\T)$ we define the Fourier coefficients by
\bq
f_n &=& \frac{1}{2\pi} \int_0^{2\pi} f(e^{i\theta})\, e^{-i\theta}\, d\theta,\qquad n\in\Z.
\eq
The Hardy spaces $H^p(\T)$ and $\overline{H^p(\T)}$ are defined by
\bq
H^p(\T) &=& \left\{ \, f\in L^p(\T)\,:\, f_n=0 \mbox{ for all } n<0 \right\},\\
\overline{H^p(\T)} &=& \left\{ \, f\in L^p(\T)\,:\, f_n=0 \mbox{ for all } n>0 \right\},
\eq
where $1\le p \le \iy$. For $1<p<\iy$, the Riesz projection 
\bq
P: \sum_{n\in\Z} f_n t^n \in L^p(\T) \mapsto \sum_{n\ge0} f_n t^n \in L^p(\T)
\eq
is bounded and has image $H^p(\T)$. We also introduce the flip operator
\be
J:f(t)\mapsto t\iv f(t\iv),\qquad t\in\T,
\ee
acting on $L^p(\T)$, and for given $a\in L^\iy(\T)$, the (bounded)
multiplication operator
\bq\label{La}
L(a): f(t)\in L^p(\T)\mapsto a(t) f(t)\in L^p(\T).
\eq
Given these operators, we can define the Toeplitz and Hankel operator with generating function
$a\in L^\iy(\T)$ as compressions of operators acting on $L^p(\T)$ to operators acting $H^p(\T)$,
$1<p<\iy$,
\be
T(a)=PL(a)P|_{H^p(\T)},\qquad 
H(a)=PL(a)JP|_{H^p(\T)}.
\ee
With respect to the standard basis $\{t^n\}_{n=0}^\iy$, these operators have the 
matrix representations
$$
T(a)\cong (a_{j-k}),\qquad H(a)\cong (a_{j+k+1}),\qquad j,k\ge 0,
$$
given in terms of the Fourier coefficients of the symbol $a$. Let us at this point recall the important 
relations between Toeplitz and Hankel operators,
\bq\label{Tab}
T(ab)  &=& T(a)T(b) +H(a) H(\tb),\\
H(ab) &=& T(a)H(b)+ H(a) T(\tb).\label{Hab}
\eq
Therein and through out the paper, we will use the ``tilde''-notation
\bq\label{tilde}
\tb(t)=b(t\iv), \qquad t\in \T.
\eq
In terms of the Fourier coefficient this means that $\tb_n=b_{-n}$.
We will denote the upper half of the unit circle by
\be\label{Tp}
\T_+=\{ t\in \T\,:\, \Im(t)>0 \}.
\ee
The set of piecewise continuous functions on $\T$ is denoted by $PC$.

The goal of this paper is to establish an explicit Fredholm and invertibility theory for operators
of the form
\bq
T(a)+H(b)
\eq 
with $a,b\in PC$ under the assumption 
\be\label{cond}
a \ta= b\tb .
\ee
In particular, we are going to establish formulas for the computation of the defect numbers.
The results on the defect number will even go beyond the piecewise continuous class assuming only the factorability of certain auxiliary functions (see Sect.~\ref{s4}). In this sense, this paper is a generalization of \cite{BE2} (see also \cite{BE1,BE3}), where the focus were Toeplitz+Hankel operators
$T(a)+H(a)$, i.e., $a=b$.

Recall that an operator $A$ acting on a Banach space $X$ is called Fredholm if its image
is closed and if the defect numbers
$$
\alpha(A)=\ker A,\qquad \beta(A)=\dim(X/\mathrm{im}(A))
$$
are finite. In this case, the number $\alpha(A)-\beta(A)$ is called the Fredholm index of $A$.
For more information about the notion of Fredholmness we refer to \cite{GKre}.

For pure Toeplitz operators $T(a)$ (as well as for the related class of singular integral operators and Wiener-Hopf operators),
the Fredholm and invertibility theory is by now classic. For its long history we refer to the monographs \cite{BS2, BS1, D, GK}
and the references therein. For piecewise continuous symbols an explicit Fredholm and invertibility theory is well established.
Moreover, for general $L^\iy$-symbols the Fredholm and invertibility theory is related to the notion of
Wiener-Hopf factorization (see \cite{CG, LS} and the papers \cite{S, W}).

The situation is more complicated for Toeplitz+Hankel operators $T(a)+H(b)$ (as well as for singular integral operators with flip and Wiener-Hopf-Hankel operators). For piecewise continuous symbols, explicit criteria for the Fredholmness are in principle known (or, more, precisely,
can be derived from known result in a straightforward way). Their basis is a description of the Banach algebra of bounded linear operators on $\Hp$ generated by Toeplitz and Hankel operators with piecewise continuous symbols. It was first given by Power \cite{P} in the case $p=2$ (see also \cite[Sec.~4.95-102]{BS1}). For general $1<p<\iy$, a corresponding result was established by Roch and Silbermann \cite[Sec.~9]{RS}. 
The Roch-Silbermann result is based on localization principles and on Mellin convolution techniques.
Let us mention that  corresponding results for Toeplitz+Hankel operators on $\ell^p$-spaces can be found in \cite{RS2}, and for Wiener-Hopf-Hankel operators of $L^p(\R_+)$ in \cite{RSS}. The paper \cite{RS2} also establishes a formula for the Fredholm index. 
For more information about the theory of singular integral operators with flip (which has been advanced by many people) we refer to \cite{KS, KL, L} and the references therein.

Given the current status of the theory of Toeplitz + Hankel operators, the challenge now is the computation of the defect numbers (which is necessary to obtain an invertibility criteria).
In previous publications \cite{BE2,BE3} the authors have developed a theory for special
Toeplitz + Hankel operators $T(a)+H(a)$ (i.e., $a=b$). Moreover, one of the authors showed in \cite[Thm.~6.29]{E2} (see also \cite{E1} for the case of continuous symbols) that $T(a)+H(b)$
with $a,b\in L^\iy(\T)$ is Fredholm if and only if the function
$$
\Phi=\left(\ba{cc} b \ta\iv & a-b\ta\iv \tb \\ \ta\iv & -\ta\iv \tb \ea\right)
$$
admits a certain type of antisymmetric factorization. Moreover, the defect numbers can be computed from the partial indices. While this clarifies the connection between invertibility and factorization theory, the caveat is that for matrix symbols (such as $\Phi$) the factorization and the computation of the partial indices is a problem that is believed not to have an explicit solution in general. Hence, for arbitrary $a,b$ (even if they are continuous) there is little hope to get a truely explicit invertibility criteria for $T(a)+H(b)$.

On the other hand, if the matrix function $\Phi$ is triangular with piecewise continuous entries, then 
there exist a explicit factorization procedures \cite{CG, LS}, which would provide a way to determine the defect numbers of $T(a)+H(b)$.
Notice that the condition $a\ta=b\tb$ makes $\Phi$ triangular. This idea has been pursued in the special case of 
$T(a)+H(t\iv a)$, i.e., $b(t)=t\iv a(t)$, in \cite[Thm.~6.33]{E2}, which relies on \cite[Thm.~6.29]{E2}. The proofs of both theorems are quite involved. For this reason we will not follow this strategy. Instead we will give in this paper a more self-contained derivation, which follows more closely our previous publications
\cite{BE2, BE3}. Therein we encountered the factorization of a scalar symbol (rather than a matrix symbol).

The triangular form of $\Phi$ gives one motivation for considering the condition $a\ta=b\tb$.
Another motivation is that the case of  $a\ta=b\tb$ has several important subclasses. Besides the already mentioned operators $T(a)+H(a)$ and the related $T(a)-H(a)$ (which can be reduced to the first case), one also encounters
$$
T(a) - H(t\iv a) \quad \mbox{ and }\quad  T(a)+H(t a).
$$
The invertibility theory of these four classes of operators turns out to be particularly simple. Indeed, in case of smooth symbols sufficient criteria have been established in \cite{BE4} (see Prop.~3.1 and 4.4 therein). 
It should be mentioned that the determinants of the finite sections of the above four classes of operators are important in the study of random matrix ensembles (see \cite{BCE,BE4} for more information).

Another class of importance is that of operators $I+H(\psi)$ where $\psi$ satisfies $\psi\tilde{\psi}=1$.
For very particular symbols $\psi$, the invertibility of $I+H(\psi)$ has already been proved ($p=2$) and 
was needed for other results \cite{BE5,E3,E4,E5}, which were related to the asymptotics of determinants and random matrix theory.

An outline of the paper is as follows:
In Section \ref{s2} we will first recall the (more general) results of \cite{RS} and then specialize in order to
derive a Fredholm criteria for $T(a)+H(b)$ with $a\ta=b\tb$, $a,b\in PC$. In the Fredholm criteria we will encounter two auxiliary functions $c=a/b=\tb/\ta$ and $d=\ta/b=\tb/a$. Furthermore, the Fredholm index
of $T(a)+H(b)$ can be expressed in terms of winding numbers of curves generated by the images of $c$ and $d$. The Fredholm criteria and the index formula depend of course on the underlying space $H^p(\T)$.

In Section \ref{s3} we show that assuming Fredholmness of $T(a)+H(b)$ and $a,b\in PC$, the functions $c,d$ admit
a certain kind of Wiener-Hopf factorization (which is antisymmetric in the sense of \cite{BE3, E2}).
One could ask whether this kind of result can be extended beyond the piecewise continuous class.
That this is not the case will be shown by a simple example at the end of the section.

In Section \ref{s4} we are going to establish formulas for the defect numbers of $T(a)+H(b)$.
This will be done under the following assumptions: $T(a)+H(b)$ is Fredholm (where $a,b\in L^\iy(\T)$ 
and $a\ta=b\tb$) and the auxiliary functions $c,d$ admit antisymmetric factorizations. In these factorizations there occur indices $n,m\in\Z$ which will determine the defect numbers.
In some cases ($n,m>0$), the formulas will also depend on the factors of the factorizations.
As a corollary, an invertibility criteria is derived.
Making the connection back to the piecewise continuous case, the existence of the factorizations of $c,d$ (which was shown in Sect.~\ref{s3}) implies that the formulas for the defect numbers are always applicable. Moreover, the indices $n,m$ relate to the winding numbers associated 
with $c$ and $d$.

Finally, in Section \ref{s5} we will focus on special cases. First we consider the four classes of operators $T(a)\pm H(a)$, $T(a)-H(t\iv a)$, and $T(a)+H(ta)$. The nice thing about these cases is that the factors in the factorization will never play a role and that in addition only one index (or winding number) occurs. As a second illustration we will consider operators $I+H(\psi)$ with $\psi\in PC$, $\psi\tilde{\psi}=1$. Here the criteria are slightly more complicated.

%%%%%%%%%%%%%%%
%%%%%%%%%%%%%%%

\section{Fredholm theory for piecewise continuous functions}
\label{s2}

As already pointed out in the introduction, the goal of this section is to state explicit 
criteria for the Fredholmness of $T(a)+H(b)$ on $\Hp$ in the case $a\ta=b\tb$, $a,b\in PC$.
In addition, we will derive a formula for the Fredholm index. 

The Fredholm criteria is a straightforward consequence of results in \cite{RS}, which we are going to 
state first. For this purpose we have to introduce some notations.
Throughout this paper we will assume that $1<p<\iy$ and $1/p+1/q=1$.
Recall that the Mellin transform $\cM:L^p(\Rp)\to L^p(\R)$ is defined by
\bq
(\cM f)(x) &=& \int_0^\iy \xi^{-1+1/p-ix}f(\xi)\,d\xi,\quad x\in\R,
\eq
and that for $\phi\in L^\iy(\R)$ the Mellin convolution operator
$M^0(\phi)\in\cL(L^p(\Rp))$ is given by
\bq
M^0(\phi)f &=& \cM\iv(\phi(\cM f)).
\eq
Let $S$ denote the singular integral operator acting on the space $L^p(\Rp)$,
\bq
(Sf)(x) &=& \frac{1}{\pi i}\int_0^\iy \frac{f(y)}{y-x}\,dy,\quad x\in\Rp,
\eq
where the singular integral has to be understood as the Cauchy principal value,
and let $N$ stand for the Hankel-type integral operator acting on $L^p(\Rp)$ by
\bq
(Nf)(x) &=& \frac{1}{\pi i}\int_0^\iy \frac{f(y)}{y+x}\,dy,\quad x\in\Rp.
\eq
It is well known that $S$ and $N$ can be expressed as Mellin convolution
operators
\be
S=M^0(s), \qquad N=M^0(n)
\ee
with the generating functions
\be
s(x)=\coth((x+i/p)\pi),\quad
n(x)=(\sinh((x+i/p)\pi))\iv ,\quad x\in\R
\ee
(see, e.g., \cite[Sec.~2]{RS} or \cite[Chap.~2]{HRS}). Notice that $s(x)$ and $n(x)$ are continuous
on $\R$ and possess the limits $s(\pm\iy)=\pm1$ and $n(\pm\iy)=0$ for
$x\to\pm\iy$. Moreover, $s^2-n^2=1$.

Let $S_\T$ stand for the singular integral operator acting on $L^p(\T)$,
\bq
(S_\T f)(t) &=& \frac{1}{\pi i} \int_\T \frac{f(s)}{s-t}ds, \quad t\in\T,
\eq
and let $W$ stand for the ``flip'' operator 
\bq
(W f)(t) &=& f(t\iv)
\eq
acting on $L^p(\T)$.  For $\phi\in PC$ the one-sided limits are denoted by
\bq
\phi^\pm(t) &=& \lim\limits_{\varepsilon\to\pm0}\phi(te^{i\varepsilon}), \qquad t\in\T.
\eq
As before let $L(\phi)$ stand for the multiplication operator (\ref{La})
on $L^p(\T)$ with the generating function $\phi$,  and let $\T_+$ stand for the upper half of the unit circle
(\ref{Tp}).

The smallest closed subalgebra of $\cL(L^p(\T))$, which contains the operators
$S$, $W$ and $L(\phi)$ for $\phi\in PC$ will be denoted by
$\cS^p(PC)$.
Specializing Theorem 9.1 of \cite{RS} to the situation where we are 
interested in, we obtain the following result. 

\pagebreak[2]
\begin{theorem}[Roch, Silbermann]
Let $1<p<\iy$. 
\label{t2.1}\
\begin{itemize}
\item[(a)]
For $\tau\in\{-1,1\}$, there exists a homomorphism 
$H_\tau:\cS^p(PC)\to \cL((L^p(\T))^2)$, which acts on the generating 
elements as follows:
\bq
H_\tau (S_\T) &=& 
\left(\ba{cc} S&-N\\ N&-S\ea\right),\qquad
H_\tau (W) \;\;=\;\;
\left(\ba{cc} 0&I\\ I&0\ea\right),
\nn\\[2ex]
H_\tau (L(\phi))  &=& 
\diag(\phi^+(\tau)I,\phi^-(\tau)I).\nn
\eq
\item[(b)]
For $\tau\in\T_+$, there exists a homomorphism 
$H_\tau:\cS^p(PC)\to \cL((L^p(\T))^4)$, which acts on the generating 
elements as follows:
\bq
H_\tau (S_\T) &=& 
\left(\ba{cccc} S&-N&0&0\\ N&-S&0&0\\ 0&0&S&-N\\0&0&N&-S\ea\right),\qquad
H_\tau (W) \;\;=\;\;
\left(\ba{cccc} 0&0&0&I\\ 0&0&I&0\\ 0&I&0&0\\ I&0&0&0\ea\right),
\nn\\[2ex]
H_\tau (L(\phi))  &=& 
\diag(\phi^+(\tau)I,\phi^-(\tau)I,\phi^+(\bar{\tau})I,\phi^-(\bar{\tau})I).\nn
\eq
\item[(c)]
An operator $A\in \cS^p(PC)$ is Fredholm if and only if the operators
$H_\tau (A)$ are invertible for all $\tau\in\{-1,1\}\cup\T_+$.
\end{itemize}
\end{theorem}

We also state the following basic result for symbols in $L^\iy(\T)$. 
The proof is standard (see, e.g.,  Prop.~2.2 in \cite{BE2}).

\begin{proposition}\label{p2.2}
Let $a,b\in L^\iy(\T)$ and $1<p<\iy$. If $T(a)+H(b)$ is Fredholm on $\Hp$, then $a$ is invertible in $L^\iy(\T)$.
\end{proposition}

Since we are interested in Fredholm criteria, in view of this result it is no loss of generality to assume $a\in GPC$, where $GPC$ stands for the group of all invertible elements in the algebra $PC$ of piecewise continuous functions on $\T$. Moreover, due to the assumption $a\ta=b \tb$ we can also assume $b\in GPC$.

As consequence of Theorem \ref{t2.1} we now obtain explicit Fredholm criteria for $T(a)+H(b)$ on $\Hp$ with $a,b\in GPC$, $a\ta=b\tb$.

\begin{corollary}\label{c2.2}
Let $a,b\in GPC$ satisfy $a\ta=b\tb$, and let $1<p<\iy$, $1/p+1/q=1$. Put
\be\label{f.c9}
c=\frac{a}{b}=\frac{\tb}{\ta},\qquad d=\frac{\ta}{b}=\frac{\tb}{a}.
\ee
Then $T(a)+H(b)$ is Fredholm on $\Hp$ if and only if
\bq\label{f.c10}
\frac{1}{2\pi}\arg c^-(1) \notin \frac{1}{2}+\frac{1}{2p}+\Z, &&
\frac{1}{2\pi}\arg d^-(1) \notin \frac{1}{2}+\frac{1}{2q}+\Z,
\\ \label{f.c11}
\frac{1}{2\pi}\arg c^-(-1) \notin \frac{1}{2p}+\Z,\quad&&
\frac{1}{2\pi}\arg d^-(-1) \notin \frac{1}{2q}+\Z,
\\ \label{f.c12}
\frac{1}{2\pi}\arg\left(\frac{c^-(\tau)}{c^+(\tau)}\right)  \notin \frac{1}{p}+\Z,\quad &&
\frac{1}{2\pi}\arg\left(\frac{d^-(\tau)}{d^+(\tau)}\right) \notin \frac{1}{q}+\Z,\quad 
\mbox{ for each }\tau\in\T_+.
\eq
\end{corollary}
\begin{proof}
First of all notice that $c,d\in GPC$ are well defined and satisfy 
the condition $c\tilde{c}=d\tilde{d}=1$. Moreover,
\be\label{f.14cd}
c^\pm(\tau)=\frac{a^\pm(\tau)}{b^\pm(\tau)}=\frac{b^{\mp}(\bt)}{a^{\mp}(\bt)},
\qquad
d^\pm(\tau)=\frac{a^\mp(\bt)}{b^\pm(\tau)}=\frac{b^{\mp}(\bt)}{a^{\pm}(\tau)}.
\ee
Now observe that $T(a)+H(b)$ is Fredholm on $\Hp$ if and only if
$$
A=PL(a)+PL(bt\iv)W+Q
$$
is Fredholm on $\Lp$, where $Q=I-P$ and $J=L(t\iv)W$. Indeed,
$$
PL(a)P+PL(bt\iv)WP+Q=\Big(I-PL(a)Q-PL(bt\iv)WQ\Big)\Big(PL(a)+PL(bt\iv)W+Q\Big),
$$
where the first operator on the right is invertible. Noting that $P=(I+S_{\T})/2$ and $Q=(I-S_{\T})/2$, it follows that the operator $A$ belongs
to $\cS^p(PC)$, and thus we can apply Theorem \ref{t2.1}. We conclude that
$A$ is Fredholm on $L^p(\T)$ if and only if for all $\tau\in\T_+\cup
\{-1,1\}$ the operators $H_\tau(A)$ are invertible.
Therefore we need to compute the operators $H_\tau(A)$ and examine their invertibility.

For $\tau\in\{-1,1\}$ we obtain 
$$
H_\tau(P) =\frac{1}{2}
\left(\ba{cc} I+S&-N\\ N&I-S\ea\right),\qquad
H_\tau(Q) =\frac{1}{2}
\left(\ba{cc} I-S&N\\ -N&I+S\ea\right),
$$
and 
$$
H_\tau(L(a)+L(bt\iv)W) = 
\left(\ba{cc} a^+(\tau)&\tau b^+(\tau)\\ \tau b^-(\tau)&a^-(\tau)\ea\right).
$$
It is now obvious that $H_\tau(A)$ is a $2\times 2$ Mellin convolution operator with 
a $2\times 2$ matrix symbol. Moreover, this operator is invertible if and only if the determinant of its symbol does not vanish on $\xR$. By straightforward computations (using $s^2=1+n^2$) the determinant evaluates to
$$
\frac{1}{2}\Big(a^-(\tau)(1-s(x))+a^+(\tau)(1+s(x))+\tau (b^+(\tau)-b^-(\tau))n(x)\Big).
$$
For $x=\pm\iy$ we obtain the conditions $a^\pm(\tau)\neq0$, which are fulfilled because $a\in GPC$. Using (\ref{f.14cd}) this expression becomes
$$
\frac{b^-(\tau)n(x)}{2}\Big(c^-(\tau)\frac{1-s(x)}{n(x)}+d^-(\tau)\frac{1+s(x)}{n(x)}+\tau(c^-(\tau)d^-(\tau)-1)\Big),
$$
which (up to the nonzero factor in front of it) is equal to
\bq
\lefteqn{-c^-(\tau)e^{-(x+i/p)\pi}+d^-(\tau)e^{(x+i/p)\pi}+\tau c^-(\tau)d^-(\tau)-\tau}\nn\\
&=& \tau\Big(c^-(\tau)+\tau e^{(x+i/p)\pi}\Big)\Big(d^-(\tau)-\tau e^{-(x+i/p)\pi}\Big).
\eq
This expression is non-zero for all $x\in\R$ if and only if conditions (\ref{f.c10}) and (\ref{f.c11}) hold.

For $\tau\in\T_+$ we can do a similar computation, which leads to a Mellin convolution operator
with a $4\times 4$ matrix function, whose determinant evaluates to $1/4$ times
$$
\Big(a^-(\bt)a^+(\tau)+a^+(\bt)a^-(\tau)+(b^+(\tau)-b^-(\tau))(b^+(\bt)-b^-(\bt))\Big)(1-s^2(x))
$$
$$
\mbox{}+a^-(\tau)a^-(\bt)(1-s(x))^2+a^+(\tau)a^+(\bt)(1+s(x))^2.
$$
Using $a^+(\tau)a^-(\bt)=b^+(\tau)b^-(\bt)$ and $a^-(\tau)a^+(\bt)=b^-(\tau)b^+(\bt)$ this simplifies to
$$
(b^+(\tau)b^+(\bt)+b^-(\tau)b^-(\bt))(1-s^2(x))
+a^-(\tau)a^-(\bt)(1-s(x))^2+a^+(\tau)a^+(\bt)(1+s(x))^2.
$$
For $x=\pm\iy$ we get the conditions $a^\pm(\tau)a^\pm(\bt)\neq0$, which are satisfied because $a\in GPC$. Dividing by $1-s^2(x)$ and $b^-(\tau)b^-(\bt)$ we obtain
\bq
\lefteqn{
c^-(\tau)c^-(\bt)d^-(\tau)d^-(\bt)+1+
c^-(\tau)c^-(\bt)\frac{1-s(x)}{1+s(x)}+
d^-(\tau)d^-(\bt)\frac{1+s(x)}{1-s(x)}}
\hspace*{8ex}\nn\\
&=&
\Big(c^-(\tau)c^-(\bt)-e^{2(x+i/p)\pi}\Big)
\Big(d^-(\tau)d^-(\bt)-e^{-2(x+i/p)\pi}\Big).\nn
\eq
This expression is non-zero for all $x\in\R$ if and only if conditions (\ref{f.c12}) hold.
\end{proof}

There is a more geometric interpretation of the Fredholm conditions (\ref{f.c10})--(\ref{f.c12}), which eventually suggest a formula for the Fredholm index of $T(a)+H(b)$. For $z_1,z_2\in\C$ and
$\theta\in(0,1)$ introduce the (open) circular arc joining $z_1$ and $z_2$,
$$
\cA(z_1,z_2; \theta)=\Big\{\; z\in\C\setminus\{z_1,z_2\}\;:\; \frac{1}{2\pi}\arg\left(\frac{z-z_1}{z-z_2}\right)\in \theta+\Z\Big\}.
$$ 
In case $\theta=1/2$ the arc becomes a line segment. For $z_1=z_2$
the set is empty.
Assuming $a,b\in GPC$, $a\ta=b\tb$, and introducing the auxiliary functions $c,d$ by (\ref{f.c9})
the conditions  (\ref{f.c10})--(\ref{f.c12}) can be restated by requiring that
the origin does not lie on any of the arcs
\be
\cA(1,c^+(1); \tfrac{1}{2}+\tfrac{1}{2p}),\qquad \cA(c^-(\tau),c^+(\tau);\tfrac{1}{p}),\qquad \cA(c^-(-1),1;\tfrac{1}{2p})
\ee
and
\be
\cA(1,d^+(1); \tfrac{1}{2}+\tfrac{1}{2q}),\qquad \cA(d^-(\tau),d^+(\tau);\tfrac{1}{q}),\qquad \cA(d^-(-1),1;\tfrac{1}{2q})
\ee
where $\tau\in\T_+$. Clearly we need only to look at those $\tau\in\T_+$ at which there is a jump discontinuity. Notice that $c^\pm(\tau)\neq0$ and $d^\pm(\tau)\neq 0$ because by assumption $c,d\in GPC$.

Since $c\tc=d\td=1$, the functions $c$ and $d$ are determined by their values on $\T_+$. If we let run $\tau$ along $\T_+$ from $1$ to $-1$, then the image $c(\tau)$
describes a curve, with possible jump discontinuities, starting from $c^+(1)$ and terminating at $c^-(-1)$.
Let us now insert the arc $\cA(c^-(\tau),c^+(\tau);\tfrac{1}{p})$ for each $\tau\in\T_+$ at which $c$ has a jump discontinuity. In addition, issert the arcs  $\cA(1,c^+(1); \tfrac{1}{2}+\tfrac{1}{2p})$ and $\cA(c^-(-1),1;\tfrac{1}{2p})$, if necessary, joining the end points $c^+(1)$ and $c^-(-1)$ to $1$. We will obtain a closed oriented curve $c^{\#,p}$ which does not pass through the origin if $T(a)+H(b)$ is Fredholm. To this curve we can associate a winding number, denoted  by
\be
\mathrm{wind}(c^{\#,p})\in\Z.
\ee
The construction is clear in case of finitely many jump discontinuities and piecewise smooth $c$.
For general piecewise continuous $c$ one needs to elaborate a little bit more on some standard technical details, which we will not pursue. An alternative way of defining the winding number  can be based on Lemma \ref{l2.4} below (see also Theorem \ref{t2.5}).

%%%%%%%%%%%%%%%%%%
%
%

%%BoundingBox: 81 3 322 244
%%HiResBoundingBox: 81.1875 3.1875 321.188 243.188
%%  Size 240 x 240
%%  scale .7 --> 168 x 168

\begin{figure}
% parameters h t b p !
\begin{picture}(450,578)
\put(0,15){\framebox(168,168){}}
\put(0,15){\includegraphics[scale=.7]{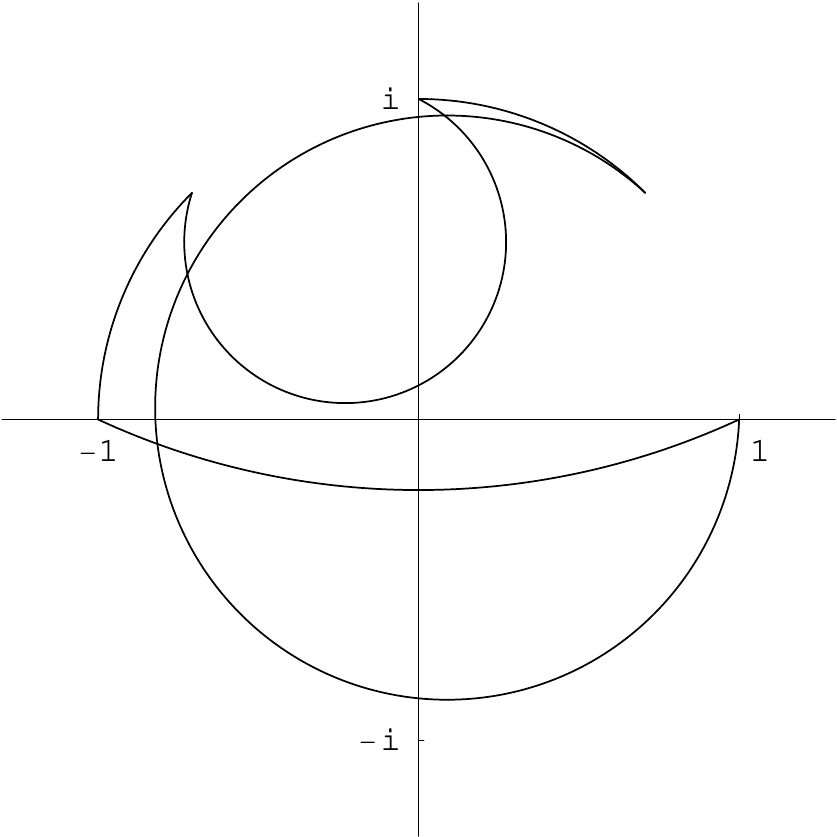}}
\put(0,0){\mbox{Fig.~5: $p=1.16$, wind$(c^{\#,p})=0$}}
\put(235,15){\framebox(168,168){}}
\put(235,15){\includegraphics[scale=.7]{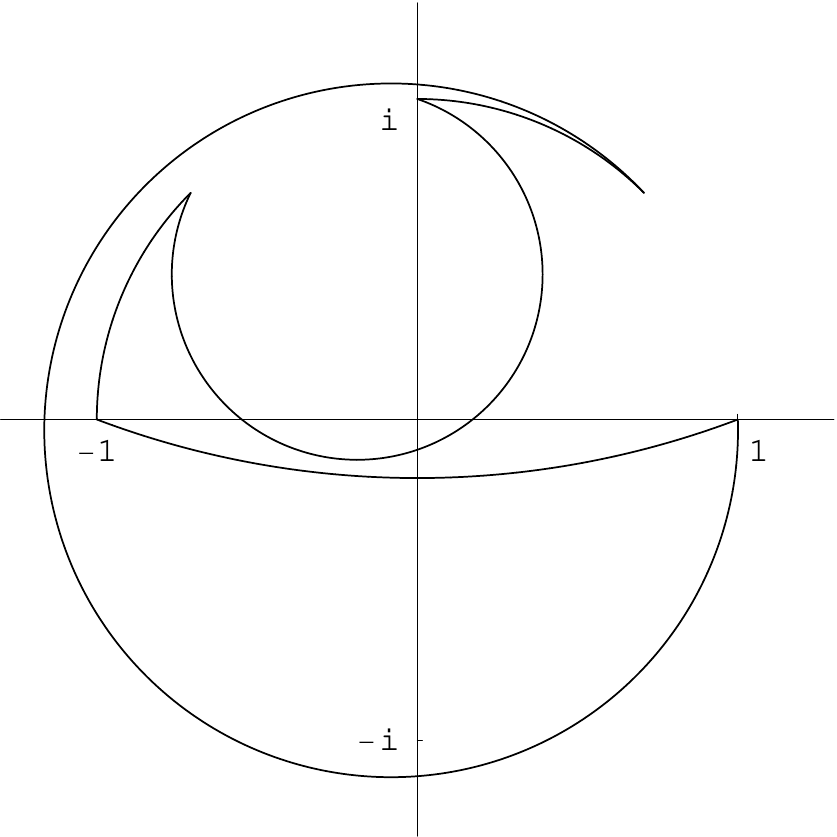}}
\put(235,0){\mbox{Fig.~6: $p=1.13$, wind$(c^{\#,p})=-1$}}
\put(0,210){\framebox(168,168){}}
\put(0,210){\includegraphics[scale=.7]{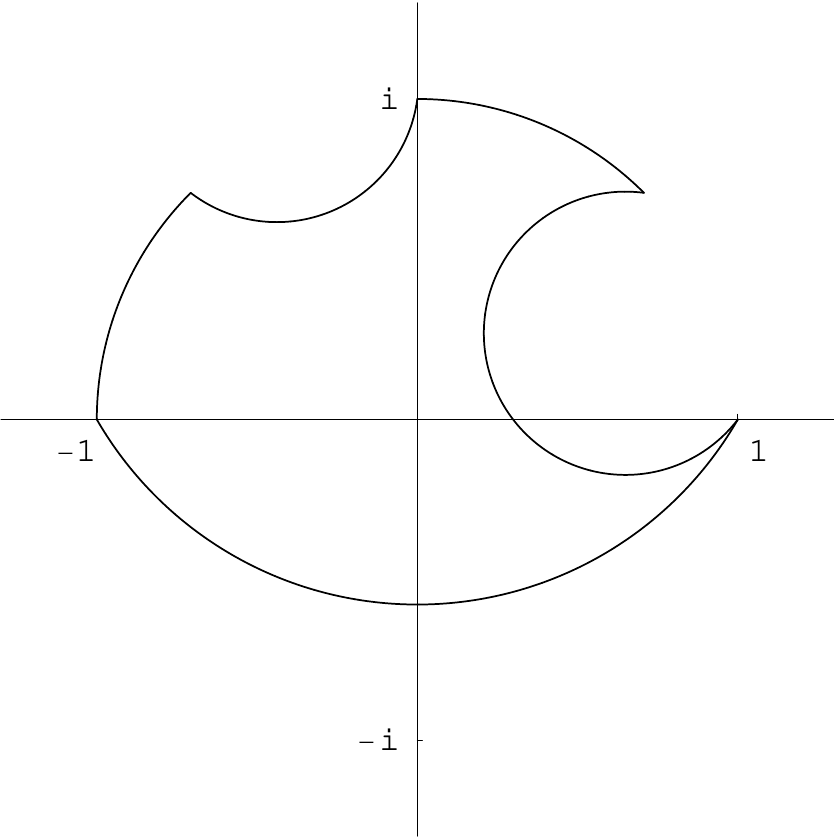}}
\put(0,195){\mbox{Fig.~3: $p=1.5$, wind$(c^{\#,p})=1$}}
\put(235,210){\framebox(168,168){}}
\put(235,210){\includegraphics[scale=.7]{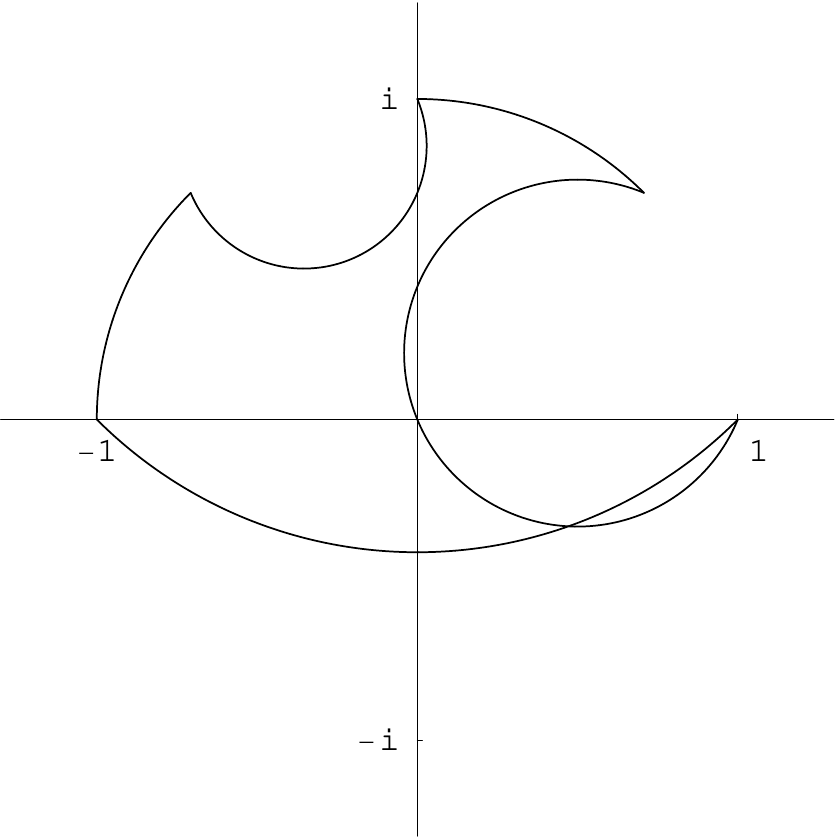}}
\put(235,195){\mbox{Fig.~4: $p=4/3$}}
\put(0,405){\framebox(168,168){}}
\put(0,405){\includegraphics[scale=.7]{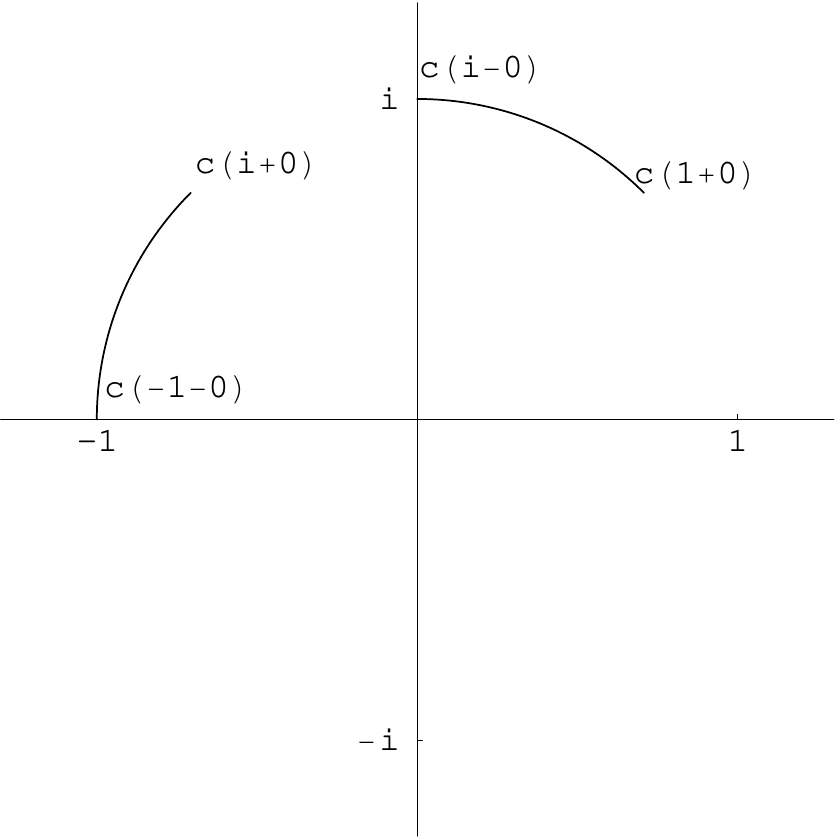}}
\put(0,390){\mbox{Fig.~1: Image of $c(e^{ix})$, $0<x<\pi$}}
\put(235,405){\framebox(168,168){}}
\put(235,405){\includegraphics[scale=.7]{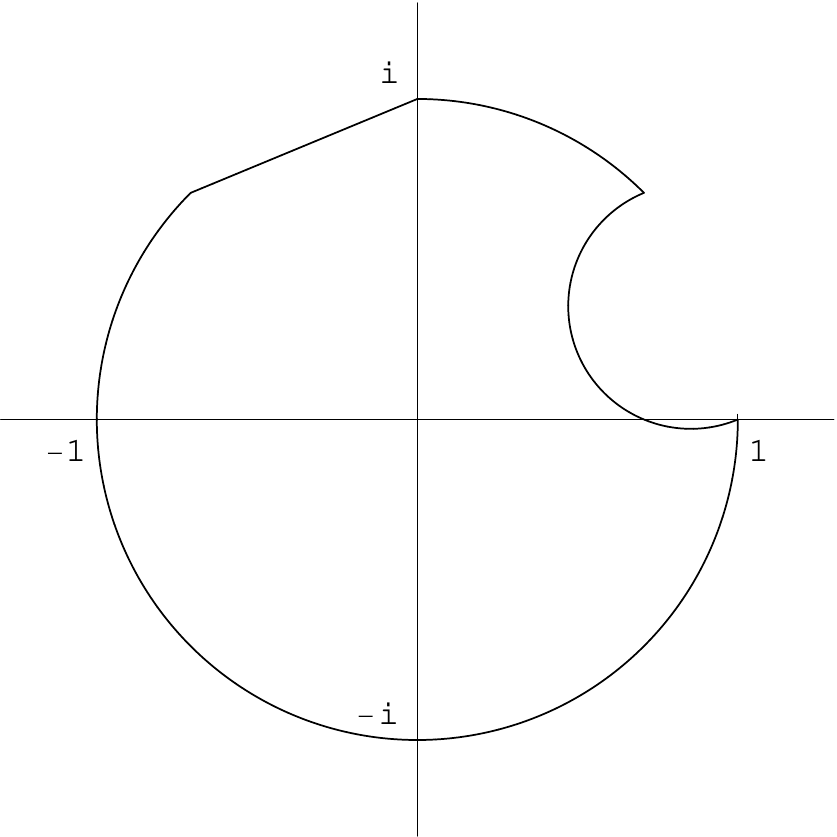}}
\put(235,390){\mbox{Fig.~2: $p=2$, wind$(c^{\#,p})=1$}}
\end{picture} 
\end{figure}
%
%% END
%%%%%%%

\begin{example}\label{ex.1}\em
Let us consider the function
\be\label{fct.c}
c(e^{ix})=\left\{\ba{ll}
e^{\pi i/4+ix/2} & 0<x<\pi/2\\
e^{\pi i/2+ix/2} & \pi/2<x<3\pi/2\\
e^{3 \pi i/4+ix/2} & 3\pi/2<x<2\pi.
\ea\right.
\ee
Clearly, $c\in GPC$ 	and $c\tc=1$. This function has jump discontinuities at $t=1$ and $t=\pm i$.
The image of this function restricted to $\T_+$ is depicted in Fig.~1. The values that determine the arcs
that have to be inserted are
$$
c^+(1)=e^{\pi i/4},\quad c^-(i)=e^{\pi i/2},\quad c^+(i)=e^{3 \pi i/4},\quad c^-(-1)=-1.
$$
The arcs are
$$
\cA(1,e^{\pi i/4};\tfrac{1}{2}+\tfrac{1}{2p}),\quad
\cA(e^{\pi i/2},e^{3 \pi i/4}; \tfrac{1}{p}),\quad
\cA(-1,1;\tfrac{1}{2p}).
$$
Notice that we also have an arc related to $t=-1$ even though $c(t)$ is continuous there.
In Fig.~2-6 we see the curves $c^{\#,p}$ for various values of $p$. For $p=2$ we have to insert a line segment and two half circles. The curve $c^{\#,p}$ passes through the origin for the values $p=4/3$ and
$p=8/7=1.1428\dots$ (not depicted). In all other cases one can associated a winding number.
\qed
\end{example}

For the function $d$ the construction is similar and yields a curve $d^{\#,q}$ with a winding number
\be
\mathrm{wind}(d^{\#,q})\in\Z,
\ee
where $1/p+1/q=1$. We can conclude that $T(a)+H(b)$ is Fredholm on $\Hp$ if and only the origin does not lie on the curves $c^{\#,p}$ and $d^{\#,q}$.

In order to relate the winding numbers of $c^{\#,p}$ and $d^{\#,q}$ to the Fredholm index
of $T(a)+H(b)$ we need the following lemma. 

\begin{lemma}\label{l2.4}
Let $c,d\in GPC$, $c\tc=d\td=1$, satisfy the conditions (\ref{f.c10})--(\ref{f.c12}). Then there exist
$n,m\in \Z$ and $\gamma,\delta\in PC$ such that  $\gamma+\tilde{\gamma}=\delta+\tilde{\delta}=0$,
\be\label{f.c30}
c(t)=t^{2n}\exp(\gamma(t)),\qquad d(t)=t^{2m}\exp(\delta(t)),
\ee
and, for each $\tau\in\T_+$,
\bq\label{f.c31}
-\tfrac{1}{2\pi}\Im \gamma^+(1) \in \left(-\tfrac{1}{2q},\tfrac{1}{2}+\tfrac{1}{2p}\right), \qquad\quad&&
-\tfrac{1}{2\pi}\Im \delta^+(1) \notin \left(-\tfrac{1}{2p}, \tfrac{1}{2}+\tfrac{1}{2q}\right),
\\ \label{f.c32}
\tfrac{1}{2\pi}\left[
\Im \gamma^-(\tau)-\Im \gamma^+(\tau)\right]  \in \left(-\tfrac{1}{q},\tfrac{1}{p}\right),\,&&
\tfrac{1}{2\pi}\left[\Im \delta^-(\tau)-\Im \delta^+(\tau)\right] \in \left(-\tfrac{1}{p},\tfrac{1}{q}\right),\quad 
\\ \label{f.c33}
\tfrac{1}{2\pi}\Im \gamma^-(-1) \in \left(-\tfrac{1}{2}-\tfrac{1}{2q},\tfrac{1}{2p}\right),\qquad\quad&&
\tfrac{1}{2\pi}\Im \delta^-(-1) \in\left(-\tfrac{1}{2}-\tfrac{1}{2p}, \tfrac{1}{2q}\right).
\eq
The representations (\ref{f.c30}) are unique.
\end{lemma}
Notice that all the intervals above are open and have length one.
\begin{proof}
Let us consider $c$ on $\T_+\cup\{1,-1\}$. Since $c$ satisfies (\ref{f.c10})--(\ref{f.c12}) it is possible to define a logarithm of $c$ on $\T_+\cup\{1,-1\}$, i.e., $c(t)=\exp(\gamma_0(t))$ such that 
(\ref{f.c31}), (\ref{f.c32}), and 
$$
\tfrac{1}{2\pi}\Im \gamma_0^-(-1) \notin \tfrac{1}{2p}+\Z,\qquad
$$
holds. This can be done by defining the logarithm first locally (up to a multiple of $2\pi i$) and then 
``gluing'' it together using a finite open covering of $\T_+\cup\{1,-1\}$. Even though (\ref{f.c33}) may not hold, it follows that there exists a (unique) $n$ such that 
$$
\tfrac{1}{2\pi}\Im \gamma_0^-(-1)  \in n+ \left(-\tfrac{1}{2}-\tfrac{1}{2q},\tfrac{1}{2p}\right).
$$
Now we define 
$$
\gamma(e^{ix})=\gamma_0(e^{ix}) - 2inx, \qquad 0<x<\pi,
$$
and extend $\gamma$ via $\gamma+\tilde{\gamma}=0$ to all of $\T$, i.e., $\gamma(e^{-ix}):=-\gamma(e^{ix})$ for $0<x<\pi$. It is easy to verify  that 
(\ref{f.c30}) and (\ref{f.c31})--(\ref{f.c33}) hold. 

As for  the uniqueness of the representation, assume that we have two  different representations with
$\gamma_1, \gamma_2$ and $n_1,n_2$. Put $n=n_1-n_2$ and $\gamma=\gamma_1-\gamma_2$ to conclude that $t^{-2n}=\exp(\gamma(t))$. The restrictions (\ref{f.c31})--(\ref{f.c33}) for $\gamma_1$ and $\gamma_2$ imply that 
$$
\tfrac{1}{2\pi}\Big|\Im \gamma^\pm(\pm1)\Big|<1,\qquad\tfrac{1}{2\pi}\Big|\Im\gamma^-(\tau)-\Im\gamma^+(\tau)\Big|<1
$$
due to the length of the original intervals. Using  $t^{-2n}=\exp(\gamma(t))$, it follows that 
$\gamma^\pm(\pm1)=0$  and that $\gamma(t)$ is continuous for $t\in\T_+$. Continuity implies the
uniqueness of the logarithm $\gamma$ and $\gamma(e^{ix})=-inx$, which yields $n=0$ and $\gamma\equiv 0$.

The argumentation for $d, \delta$, and $m$ is similar.
\end{proof}

\begin{theorem}\label{t2.5}
Let $a,b\in GPC$ satisfy $a\ta=b\tb$, and let $1<p<\iy$, $1/p+1/q=1$. Define $c,d$ by (\ref{f.c9}) and assume 
(\ref{f.c10})--(\ref{f.c12}). Then $T(a)+H(b)$ is Fredholm on $\Hp$ and has Fredholm index
\be\label{f34}
\ind(T(a)+H(b))=\wind(d^{\#,q})- \wind (c^{\#,p}).
\ee
Morever, assuming (\ref{f.c30}) we have $n=\wind (c^{\#,p})$ and $m=\wind (d^{\#,q})$.

\end{theorem}
\begin{proof}
{}From Corollary \ref{c2.2} we know that $T(a)+H(b)$ is Fredholm.
Moreover, we can apply Lemma \ref{l2.4} and conclude that $c$ and $d$ can be written
in the form (\ref{f.c30}). 

Our first goal is to show that 
\be\label{f.35}
n=\wind (c^{\#,p}), \qquad
m=\wind (d^{\#,q}).
\ee
For this purpose we introduce (for a parameter $\lambda\in[0,1]$) the functions
$$
c_\lambda(t)=t^{2n} \exp(\lambda\gamma(t)),\qquad
d_\lambda(t)=t^{2m} \exp(\lambda\delta(t)).
$$
Clearly, $c_\lambda\tc_\lambda=d_\lambda\td_\lambda=1$ because $\gamma+\tilde{\gamma}=\delta+\tilde{\delta}=0$.
{}From (\ref{f.c31})--(\ref{f.c33}) it follows that the conditions  (\ref{f.c10})--(\ref{f.c12}) hold for the functions $c_\lambda$ and $d_\lambda$ instead of $c$ and $d$. For this reason winding numbers can be associated to $c_\lambda$ and $d_\lambda$ for each $\lambda\in [0,1]$. The mappings $\lambda\mapsto c_\lambda$ and $\lambda\mapsto d_\lambda$ can be thought of as continuous deformations. This implies that 
$\wind( c_\lambda^{\#,p})$ and $\wind(d_\lambda^{\#,q})$ do not depend on $\lambda$.
Put $\lambda=0$ and $\lambda=1$ to conclude that 
$$
\wind( c^{\#,p})=\wind((t^{2n})^{\#,p}),\qquad
\wind(d^{\#,q})=\wind((t^{2m})^{\#,q}).
$$
Finally, we have $\wind((t^{2n})^{\#,p})=n$ and $\wind((t^{2m})^{\#,q})=m$, which  implies (\ref{f.35}).
Notice that the winding number is taken with respect to $\T_+$ only.

In the second part we show that $\ind(T(a)+H(b))=m-n$ and we use essentially the same deformation argument.
Introduce, for $\lambda\in[0,1]$,
$$
a_\lambda= a \exp((1-\lambda)(-\gamma+\delta)/2),\qquad
b_\lambda= b \exp((1-\lambda)(\gamma+\delta)/2).
$$
Clearly, $a_1=a$, $b_1=b$, while for all $\lambda\in[0,1]$,
$a_\lambda\ta_\lambda=b_\lambda \tb_{\lambda}$ (because of $\gamma+\tilde{\gamma}=\delta+\tilde{\delta}=0$). We are therefore in a position to 
use Corollary \ref{c2.2} in order to decide whether $T(a_\lambda)+H(b_\lambda)$ is Fredholm.
The corresponding auxiliary functions evaluate to
$$
\frac{a_\lambda}{b_\lambda}=\frac{a}{b} \exp(-(1-\lambda)\gamma)=t^{2n} \exp(\lambda \gamma)=c_\lambda
$$
and
$$
\frac{\ta_\lambda}{b_\lambda}=\frac{\ta}{b}\exp(-(1-\lambda)\delta)=t^{2m} \exp(\lambda
\delta)=d_\lambda,
$$
where $c_\lambda$ and $d_\lambda$ were already defined above.
These functions satisfy (\ref{f.c10})--(\ref{f.c12}), and therefore
$T(a_\lambda)+H(b_\lambda)$ is Fredholm on $\Hp$ for $\lambda\in[0,1]$. Because the dependence on $\lambda$ is continuous, and the Fredholm index is stable under small perturbations, it follows that the index of $T(a_\lambda)+H(b_\lambda)$ does not depend on $\lambda$. Putting $\lambda=0$ and $\lambda=1$ we can conclude that
$$
\ind(T(a)+H(b))=\ind(T(a_0)+H(b_0)).
$$
Using the above equations, it follows that 
$$
\frac{a_0}{b_0}=t^{2n},\qquad \frac{\ta_0}{b_0}=t^{2m}.
$$
Making a substitution $a_0(t)=\phi(t) t^{n-m}$, $b_0(t)=\psi(t) t^{-n-m}$, it follows that 
$\phi=\psi=\tilde{\phi}$. Hence we obtain
\bq
T(a_0)+H(b_0) &=&
T(\phi t^{n-m})+H(\phi t^{-n-m})\\
&=&
T(t^{-m})\Big(T(\phi)+H(\phi)\Big)T(t^{n})+\mbox{ compact}
\eq
by using (\ref{Tab})--(\ref{Hab}) and the fact that Hankel operators with symbols $t^{\varkappa}$
are compact. The operator $T(\phi)+H(\phi)$ is invertible with inverse $T(\phi\iv)+H(\phi\iv)$
again by (\ref{Tab})--(\ref{Hab})  and because $\phi=\tilde{\phi}\in GPC$. The Fredholm indices 
of $T(t^n)$ and $T(t^{-m})$ are $-n$ and $m$, respectively. 
Hence it follows that $\ind(T(a)+H(b))=m-n$.
\end{proof}

An alternative index formula for $T(a)+H(b)$ in the special case $a=b$
was given in \cite{BE1,BE2}. Therein the index is related to the winding number of a curve which 
arises from the image of $a(t)$ as $t$ runs along all of $\T$. Instead of arcs certain cubic curves
have to be inserted. Moreover, the curve that arises from $a$ identifies at the same time the
essential spectrum of $T(a)+H(a)$.

Our index formula in the special case $a=b$ yields the trivial $c=1$ (with $\wind(c^{\#,p})=0$) and the auxiliary function $d=\tilde{a}/a$ upon which the computation of the index is based.  In general, the curves $c^{\#,p}$ and $d^{\#,q}$ do not seem to be related to the essential spectrum of $T(a)+H(b)$.

Recently, an index formula for Toeplitz plus Hankel operators $T(a)+H(b)$ acting on $\ell^p(\Z_+)$ for general piecewise continuous $a,b$, (i.e., without the condition $a\ta=b\tb$) was established by Roch and Silbermann \cite{RS2}. It involves a matrix valued symbol
instead of the two auxiliary functions $c,d$, and it is more difficult to state and to prove.

%%%%%%%%%%%%%%%%

\section{Factorization results}
\label{s3}

The main goal of this section is to establish a factorization result for the auxiliary functions $c$ and $d$
under the assumption that $T(a)+H(b)$ is Fredholm on $\Hp$ with $a,b\in PC$, $a\ta=b\tb$.
The type of factorization is a kind of Wiener-Hopf factorization, which has appeared before in 
\cite{BE3} (see also \cite{E2}) and has been called antisymmetric factorization. The type of factorization applies to functions $\phi$ satisfying $\phi\tilde{\phi}=1$. The factorization is of the form $\phi(t)=\phi_+(t)t^{2\kappa}\phi_-(t)$ with the condition $\phi_-(t)=\tilde{\phi}\iv_+(t)$.
Because of this relationship between the factors, it was named ``antisymmetric''.

Before we start with the factorization result we will establish a (rather straightforward) auxiliary result, which shows that under the corresponding assumptions,
$c$ and $d$ can be written  as a product of the functions $u_{\tau,\beta}$ where the parameters $\beta$ satisfy a certain condition. The functions $u_{\tau,\beta}$ with $\tau\in\T$ and $\beta\in\C$
are defined by
\be\label{def.u}
u_{\tau,\beta}(\tau e^{ix}) = e^{i\beta(x-\pi)},\qquad 0<x<2\pi.
\ee
These are piecewise continuous functions with one jump at $\tau$ and $\beta$ determining its size.
Let us also mention that 
\be\label{u.p}
u_{\tau,\beta_1+\beta_2}=u_{\tau,\beta_1}u_{\tau,\beta_2}\quad \mbox{ and }\quad
u_{\tau,n}(t)= (-t/\tau)^n.
\ee

The product representation will also contain functions which are either continuous or for which the jump
discontinuities are small in a certain sense. To make this precise, consider $0< \eps\le 1/2$ and let
$\widehat{\cB}_\eps$ stand for the set of all $\phi\in GPC$ for which
\be\label{f.arg}
\left|\frac{1}{2\pi}\arg \frac{\phi^-(\tau)}{\phi^+(\tau)}\right| < \eps
\ee
for each $\tau\in\T$. For $\eps=0$ we put $\widehat{\cB}_0=GC(\T)$, the set of all non-zero continuous functions on $\T$.
The condition $\phi\in\widehat{\cB}_\eps$, $\eps\in[0,1/2]$, implies that for each $\tau\in\T$ one can find a half-plane
$H_{a}=\{z\in\C\,:\, \Re(z/a)>0\,\}$, $|a|=1$, and a neighborhood $U\subseteq\T$ of $\tau$ such that 
$\phi(t)\in H_a$ for each $t\in U$. Using a compactness argument, it is possible to define a winding number for such a function $\phi$. 

Finally, for $\eps\in[0,1/2]$, let $\cB_\eps$ stand for the set of 
all $\phi\in\widehat{\cB}_\eps$ which have winding number zero, which are continuous at $\tau=\pm1$,
and satisfy $\phi(\pm1)=1$ and $\phi\tilde{\phi}=1$.

\begin{proposition}\label{p2.4}
Let $a,b\in GPC$  satisfy  $a\ta=b\tb$, and let $1<p<\iy$, $1/p+1/q=1$. Put
\be\label{f.c9b}
c=\frac{a}{b}=\frac{\tb}{\ta},\qquad d=\frac{\ta}{b}=\frac{\tb}{a},
\ee
and assume that $T(a)+H(b)$ is Fredholm on $H^p(\T)$. Then for 
each $0<\eps\le 1/2$, the functions $c$ and $d$ can be represented in the form
\bq\label{f.13b}
c &=& t^{2n}\cdot  c_0 \cdot u_{1, 2\gamma^+}\cdot   u_{-1,2\gamma^-}
\prod_{r=1}^R u_{\tau_r,\gamma_r} \cdot  u_{\bt_r,\gamma_r},
\\    \label{f.14b}
d &=& t^{2m} \cdot  d_0 \cdot u_{1, 2\delta^+} \cdot  u_{-1,2\delta^-}
\prod_{r=1}^R u_{\tau_r,\delta_r}\cdot   u_{\bt_r,\delta_r},
\eq
with parameters satisfying 
\begin{enumerate}
\item[(i)] $-1/2q<\Re\gamma^+<1/2+1/2p$, $-1/2p<\Re \delta^+<1/2+1/2q$, 
\item[(ii)] $-1/2-1/2q<\Re\gamma^-<1/2p$, $-1/2-1/2p<\Re \delta^-<1/2q$, 
\item[(iii)] $-1/q<\Re\gamma_r<1/p$, $-1/p<\Re \delta_r<1/q$, 
\end{enumerate}
and where $n,m\in\Z$ and $c_0,d_0\in\cB_\eps$. Therein $\tau_1,\dots,\tau_R\in\T_+$ are distinct. 

If $c,d$ have only finitely many jump discontinuities, then there exist representations with $c_0,d_0\in \cB_0$.
\end{proposition}
\begin{proof}
Given $0<\eps\le 1/2$ we consider the set $S$ of all $\tau\in\T\setminus\{-1,1\}$ for which
(\ref{f.arg}) is not satisfied for $c$ or $d$. The set $S$ is finite. If $c$ and $d$ have only finitely many jump discontinuities, we take $S$ as the set of all jump discontinuities of $c$ and $d$ on $\T\setminus\{-1,1\}$.

Because $c\tc=d\td=1$, we have $S=\{\bt\,:\,\tau\in S\}$. We enumerate the elements in $S$
as follows,
$$
S=\{\tau_1,\dots,\tau_R\}\cup \{\bt_1,\dots,\bt_R\},
$$
stipulating that all $\tau_r\in \T_+$ are distinct.  Up to change of order, the $\tau_r$'s are
uniquely determined. Because of the assumption of Fredholmness, conditions (\ref{f.c10})--(\ref{f.c12}) are satisfied, and hence we can uniquely choose
parameters $\gamma$ and $\delta$ such that conditions (i)--(iii) of this proposition are fulfilled
and such that 
\begin{equation}
\ba{rclcrcl}
c^-(1) & = & e^{2\pi i \gamma^+},&\qquad&
d^-(1) & = & e^{2\pi i \delta^+},\\[1ex]
c^-(-1) & = & e^{2\pi i \gamma^-},&&
d^-(-1) & =& e^{2\pi i \delta^-},
\\[1ex]
\ds
\frac{c^-(\tau_r)}{c^+(\tau_r)} &=& e^{2 \pi i \gamma_r}, &&
\ds
\frac{d^-(\tau_r)}{d^+(\tau_r)} &=& e^{2 \pi i \delta_r},
\ea
\end{equation}
for all $1\le r \le R$.
Notice that conditions (i)--(iii) require that the real parts of the parameters belong to a certain open interval of length one. That the end-points cannot be attained follows from the Fredholm conditions.

Now define
\bq
c _1 &=& c \cdot\Big( u_{1, 2\gamma^+}\cdot   u_{-1,2\gamma^-}
\prod_{r=1}^R u_{\tau_r,\gamma_r} \cdot  u_{\bt_r,\gamma_r},
\Big)\iv,
\\
d_1 &=& d  \cdot\Big( u_{1, 2\delta^+} \cdot  u_{-1,2\delta^-}
\prod_{r=1}^R u_{\tau_r,\delta_r}\cdot   u_{\bt_r,\delta_r}\Big)\iv.
\eq
The functions $c_1$ and $d_1$ are continuous on $S\cup\{1,-1\}$. 
Moreover, $c_1(\pm 1)=d_1(\pm 1)=1$.
This follows from the computation of the one-sided limits of the functions $u_{\tau,\beta}$,
\be\label{u.pm}
u_{\tau,\beta}^{\mp}(\tau)=e^{\pm \pi i\beta},\qquad \frac{u_{\tau,\beta}^-(\tau)}{u_{\tau,\beta}^+(\tau)}=e^{2\pi i \beta}.
\ee
Because $c\tilde{c}=d\tilde{d}=1$ and $\tilde{u}_{\tau,\beta}=u_{\bt,-\beta}$, we obtain
$c_1\tilde{c}_1=d_1\tilde{d}_1=1$. In particular, the continuity of $c_1, d_1$ at $\tau_r$ implies continuity at $\bt_r$. The size of the jumps at the other points remains unchanged. Hence
$c_1,d_1$ have only small jumps, i.e., they belong to $\widehat{\cB}_\eps$. In case $c,d$ had only finitely many jumps (and with appropriate $S$),  we obtain $c_1,d_1\in\widehat{\cB}_0$.

Now the only issue that may prevent $c_1$ and $d_1$ to belong to $\cB_\eps$ (or $\cB_0$) is a non-trivial winding number. Because $c_1(1)=c_1(-1)=1$, there exists an $n\in\Z$ such that $(2\pi)\iv \arg c_1(e^{i\theta})$  has an integer increment $n$ when $\theta$ runs from $0$ to $\pi$.
Consequently, the increment of $(2\pi)\iv \arg c_1(e^{-i\theta})$ for $\theta$ from $0$ to $\pi$ is $-n$.
Hence the winding number of $c_1$ is the even number $2n$. Now we define $c_0(t)=c_1(t)t^{-2n}$, which implies that $c_0\in \cB_\eps$ (or $\cB_0$). We argue similarly for $d_0$.
\end{proof}

Let us remark that the proof of the previous proposition could have also been accomplished using Lemma \ref{l2.4}.
We preferred to give the proof involving the functions $u_{\tau,\beta}$ because the construction
is probably the one which one would use in practice. 

Furthermore, the appearance of the exponent ``$2n$'' rather that ``$n$'' for the winding number in the representations (\ref{f.13b}) and (\ref{f.14b}) might be surprising,  but it is natural and necessary. For illustration and a partial explanation we add the following examples.

\begin{example}\em
Let us discuss which representations (\ref{f.13b})  will be obtained if the function $c$ is continuous. It turns out that one has to distinguish four cases. Since $c\tc=1$, it follows $c(1)=\pm1$ and $c(-1)=\pm1$. Notice that $u_{1,1}(t)=-t$ and 
$u_{-1,-1}(t)=t\iv$ and recall (\ref{u.p}).
\begin{enumerate}
\item[(1)]
%\underline{Case 1:}
If $c(1)=c(-1)=1$, then one can write $c(t)=t^{2n} c_0(t)$ with $c_0\in \cB_0$. Here $n$ equals to the winding number of $c(t)$ when $t$ passes along $\T_+$.
\item[(2)]
%\underline{Case 2:}
If $c(1)=c(-1)=-1$, then one can write $c(t)=t^{2n} c_0(t) u_{1,1}(t)u_{-1,-1}(t)$ with $c_0\in \cB_0$. 
Thus $\gamma^+=1/2$ and $\gamma^-=-1/2$.
\item[(3)]
%\underline{Case 3:}
If $c(1)=-1$, $c(-1)=1$, then one can write $c(t)=t^{2n} c_0(t) u_{1,1}(t)$ with $c_0\in \cB_0$. 
Thus $\gamma^+=1/2$.
\item[(4)]
%\underline{Case 4:}
If $c(1)=1$, $c(-1)=-1$, then one can write $c(t)=t^{2n} c_0(t) u_{-1,-1}(t)$ with $c_0\in \cB_0$. 
Thus  $\gamma^-=-1/2$.
\end{enumerate}
We see in this example that even if the continuous function $c$ has an odd winding number, an even exponent ``$2n$'' will appear in the representation.
\qed
\end{example}

\begin{example}\em
Let us reconsider Example \ref{ex.1} with the function $c$ defined in (\ref{fct.c}).
Computation of the sizes of the jumps yields
$$
c^-(1)=e^{-\pi i/4},\qquad
\frac{c^-(i)}{c^+(i)}=e^{-\pi i/4},
$$
from which we obtain a preliminary representation
$$
c(t)= c_1(t) u_{1,-1/4}(t)u_{i,-1/8}(t) u_{-i,-1/8}(t)
$$
with continuous $c_1$ which evaluates to $c_1(t)=t$. Using $u_{-1,1}(t)=t$, we can write
$$
c(t)= u_{1,-1/4}(t)  u_{-1,1}(t)u_{i,-1/8}(t) u_{-i,-1/8}(t)
$$
with $\gamma^+=-1/8$, $\gamma^-=1/2$, and $\gamma_1=-1/8$, $\tau_1=i$. Unfortunately, condition (ii) in Proposition \ref{p2.4} is never fulfilled and thus we need further modifications using (\ref{u.p}).
We obtain
$$
c(t)= t^{2}u_{1,-1/4}(t)  u_{-1,-1}(t)u_{i,-1/8}(t) u_{-i,-1/8}(t),
$$
i.e., $\gamma^+=-1/8$, $\gamma^-=-1/2$, $\gamma_1=-1/8$, $n=1$, which is the proper representation
for $p>4/3$. Moreover, we have
$$
c(t)= u_{1,7/4}(t)  u_{-1,-1}(t)u_{i,-1/8}(t) u_{-i,-1/8}(t),
$$
i.e., $\gamma^+=7/8$, $\gamma^-=-1/2$, $\gamma_1=-1/8$, $n=0$, for $8/7<p<4/3$, and
$$
c(t)=t^{-2} u_{1,7/4}(t)  u_{-1,-1}(t)u_{i,7/8}(t) u_{-i,7/8}(t),
$$
i.e., $\gamma^+=7/8$, $\gamma^-=-1/2$, $\gamma_1=7/8$, $n=-1$, for $1<p<8/7$.
In the cases $p=4/3$ and $p=8/7$ the Fredholm conditions do not hold (see (\ref{f.c10})--(\ref{f.c12})).
\qed
\end{example}

Now we want to continue with establishing the actual factorization result, which will be based on the
product representation (\ref{f.13b}) and (\ref{f.14b}). The factorization of the terms $u_{\tau,\beta}$ 
is given by
\bq\label{u.etaxi}
u_{\tau,\beta} &=& \eta_{\tau,\beta}\cdot  \xi_{\tau,-\beta},
\eq
where 
\be
\eta_{\tau,\beta}(t) = (1-t/\tau)^{\beta},\qquad
\xi_{\tau,\beta}(t) = (1-\tau/t)^{\beta}, \qquad t\in\T\setminus\{\tau\}.
\ee
Here the branches of the power functions are chosen in such a way that 
$\eta_{\tau,\beta}(z)$ and $\xi_{\tau,\beta}(z)$ extend by analyticity into $\{z\in\C\,:\,|z|<1\}$ and $\{z\in\C\,:\,|z|>1\}\cup\{\iy\}$, respectively, with $\eta_{\tau,\beta}(0)=\xi_{\tau,\beta}(\iy)=1$.

As an auxiliary result we need the factorization of $c_0,d_0\in\cB_\eps$.

\begin{lemma}\label{l.1}
Let $\phi\in \widehat{\cB}_\eps$, $0<\eps\le 1/2$, and assume $\wind(\phi)=0$. Then we can factor
$\phi=\phi_-\phi_+$ such that 
$$
\phi_-,\phi_-\iv \in \ovl{H^\sigma(\T)},\quad \phi_+,\phi_+\iv \in H^\sigma(\T)
$$
with $\sigma=1/\eps$. If, in fact, $\phi \in \cB_\eps$, then we can choose the factors
such that $\phi_-=\tilde{\phi}_+\iv$. 
\end{lemma}
\begin{proof}
It is well known that a Toeplitz operator $T(\phi)$ with $\phi\in GPC$ is Fredholm on $\Hp$ if and only if 
$(2\pi)\iv \arg(\phi^-(\tau)/\phi^+(\tau))$ lies in the interval $(-1/q,1/p)$ for each $\tau\in\T$, where 
$1/p+1/q=1$ (see, e.g.,  \cite[Sec.~5.36-39]{BS1} or \cite{GK,LS}).
Since $\phi\in\widehat{\cB}_\eps$, this condition holds for $p=\sigma$ as well as for
$p=\hat{\sigma}$, where $1/\hat{\sigma}:=1-1/\sigma$, noting  that $\sigma\ge2$.
It follows that $T(\phi)$ is Fredholm on $H^\sigma(\T)$ as well as on $H^{\hat{\sigma}}(\T)$.
The winding number condition implies (for instance, with the help of a continuous deformation of $\phi$ to a constant, i.e., consider $\lambda\in[0,1]\mapsto (\phi(t))^\lambda\in\widehat{\cB}_{\eps}$) that $T(\phi)$ has Fredholm index zero in both cases. By Coburn's Lemma (see, e.g., \cite[Thm.~2.38]{BS1})
$T(\phi)$ is invertible on $H^\sigma(\T)$ as well as on $H^{\hat{\sigma}}(\T)$.

It is well known (see, e.g., \cite{GK,LS,S}, or \cite{W} in the case $p=2$) that if a Toeplitz operator is invertible in $\Hp$, then its symbol admits
a canonical Wiener-Hopf factorization in $\Lp$. The latter means that $\phi=\phi_-\phi_+$ with
$\phi_-\in \ovl{\Hp}$, $\phi_-\iv\in\ovl{\Hq}$, $\phi_+\in\Hq$, $\phi_+\iv\in \Hp$. 

Now we apply the previous result with $p=\sigma$ and $p=\hat{\sigma}$, and conclude that 
$\phi$ admits two (possibly different) factorizations $\phi=\phi_-^{(1)}\phi_+^{(1)}=\phi_-^{(2)}\phi_+^{(2)}$ with factors satisfying the conditions $\phi_-^{(1)}\in \ovl{H^\sigma(\T)}, (\phi_+^{(1)})\iv\in H^{\sigma}(\T)$, $(\phi_-^{(2)})\iv\in \ovl{H^\sigma(\T)}, \phi_+^{(2)}\in H^{\sigma}(\T)$,  
noting $1/\sigma+1/\hat{\sigma}=1$.
Using the standard trick and writing $(\phi_-^{(2)})\iv \phi_-^{(1)}=\phi_+^{(2)}(\phi_+^{(1)})\iv\in L^1(\T)$ (because $\sigma\ge2$), it is easily seen that the factors differ by at most a (nonzero) constant. This proves the first part of the lemma.
  
Next assume $\phi\in\cB_\eps$, i.e., $\phi=\tilde{\phi}\iv$.
If $\phi=\phi_-\phi_+$ is a factorization, then $\phi=\tilde{\phi}_+\iv \tilde{\phi}_-\iv$ is another factoriziation. By the same standard trick as above we conclude that the factors in the two factorizations are related by a nonzero constant $\gamma$,  i.e., $\phi_-=\gamma\tilde{\phi}_+\iv$. We obtain
$\phi=\gamma\phi_+\tilde{\phi}_+\iv$. Employing $\phi=\tilde{\phi}\iv$ again, it follows that 
$\gamma=\pm1$. It remains to argue that $\gamma=1$.

Because $\phi\in \cB_\eps$ has winding number zero and $\phi(\pm1)=1$, we can define
define $\psi=\phi^{1/2}\in\cB_{\eps/2}$. Now apply what we have shown so far to $\psi$ instead of $\phi$
and conclude that we can factor $\psi=\gamma'\psi_+ \tilde{\psi}_+\iv$ with 
$\gamma'=\pm1$ and $(\psi_+)^{\pm1}\in H^{2\sigma}(\T)$. Take the square to arrive at 
$$
\phi=\gamma\phi_+\tilde{\phi}_+\iv=(\psi_+)^2(\tilde{\psi}_+)^{-2}.
$$
Because $(\psi_+)^{\pm 2}\in H^\sigma(\T)$  it follows that the factors are the same up to a nonzero constant,
say, $\phi_+=\beta (\psi_+)^2$. But then $\gamma=1$, which is what we wanted.
This concludes the proof of  the lemma.
\end{proof}

Some remarks are in order. The statement of the lemma does not hold in the case $\eps=0$, i.e., 
when $\phi$ is a nonvanishing  continuous function with winding number zero. It is a well known fact that the Wiener-Hopf factors of such a function need not be bounded. However, if $\phi$ is in 
$\wh{\cB}_0$ (or in $\cB_0$), then 
$$
\phi_+,\phi_+\iv \in \bigcap_{\sigma>0} H^\sigma(\T),
$$
and similarly for $\phi_-$. On the other hand, if $\phi$ is sufficiently smooth, say it belongs to the Wiener class, then we can write down the formula
$$
\phi_+(t) =\exp\Big( \sum_{k=1}^\iy t^k [\log \phi]_k\Big),
$$
where $[\cdots]_k$ stands for the $k$-th Fourier coefficient. Notice that if $\phi\tilde{\phi}=1$ and $\phi$
has winding number zero, then $\log \phi$ is an odd function.

We are now able to present the desired factorization for piecewise continuous symbols.

\begin{theorem}\label{t3.3x}
Let $a,b\in GPC$  satisfy  $a\ta=b\tb$, and let $1<p<\iy$, $1/p+1/q=1$. Put
\be\label{f.c9c}
c=\frac{a}{b}=\frac{\tb}{\ta},\qquad d=\frac{\ta}{b}=\frac{\tb}{a},
\ee
Assume that $T(a)+H(b)$ is Fredholm on $H^p(\T)$. Then there exist factorizations
\be\label{f.f1}
c(t)=c_+(t) t^{2n} c_+\iv(t\iv),
\qquad
d(t)=d_+(t) t^{2m} d_+\iv(t\iv),
\ee
with $n,m\in\Z$  and
\be\label{f.f2}
(1+t)c_+(t) \in \Hq,\qquad  (1-t) c_+\iv(t)\in \Hp,
\ee
\be\label{f.f3}
(1+t) d_+(t)\in \Hp,\qquad
(1-t)d_+\iv(t) \in \Hq.
\ee
\end{theorem}
\begin{proof}
We may restrict out attention to the function $c$. For the function $d$ we can argue similarly. Formally,
we can obtain the results by interchanging $p$ and $q$.

Assuming Fredholmness, the conditions (\ref{f.c10})--(\ref{f.c12}) hold.
Now we can choose $\eps>0$ sufficiently small such that $\eps<\min\{1/p,1/q\}\le 1/2$
and (with appropriate choice of the argument)
$$
-\frac{1}{q}+\eps<\frac{1}{\pi}\arg c^-(1)<1+\frac{1}{p}-\eps,\qquad
-1-\frac{1}{q}+\eps<\frac{1}{\pi}\arg c^-(-1)<\frac{1}{p}-\eps,
$$ 
and
$$
-\frac{1}{q}+\eps<\frac{1}{2\pi} \arg\left( \frac{c^-(\tau)}{c^+(\tau)}\right)<\frac{1}{p}-\eps
$$
hold for all $\tau\in\T_+$. The reason is that for each $\delta>0$ there exists only finitely many
$\tau$ such that $|\arg(c^-(\tau)/c^+(\tau))|\ge \delta$. Define
$$
\frac{1}{\hat{p}}:=\frac{1}{p}-\eps,\qquad 
\frac{1}{\hat{q}}:=\frac{1}{q}-\eps
$$
and notice  that $\hat{p},\hat{q}>1$ and $1/\hat{p}+1/\hat{q}<1$.

With this $\eps$ we apply Proposition \ref{p2.4} and obtain a representation of
$c$ in the form  (\ref{f.13b}). The corresponding parameters satisfy
\begin{enumerate}
\item[(i*)] $-1/\hat{q}<2\,\Re\gamma^+<1+1/\hat{p}$,
\item[(ii*)] $-1-1/\hat{q}<2\,\Re\gamma^-<1/\hat{p}$,
\item[(iii*)] $-1/\hat{q}<\Re\gamma_r<1/\hat{p}$.
\end{enumerate}
By Lemma \ref{l.1} we can factor $c_0=c_{0,+}(\tilde{c}_{0,+})\iv$ with
$(c_{0,+})^{\pm1}\in H^{1/\eps}(\T)$. We define
$$
c_{1,+}=\eta_{1,2\gamma^+}\cdot \eta_{-1,2\gamma^-}\prod_{r=1}^R\eta_{\tau_r,\gamma_r}\cdot
\eta_{\bt_r,\gamma_r}.
$$
Using (\ref{u.etaxi}) and $\tilde{\eta}_{\tau,\beta}=\xi_{\bt,\beta}$, it follows that
$$
c_{1,+}(\tc_{1,+})\iv=u_{1,2\gamma^+}
\cdot u_{-1,2\gamma^-}\prod_{r=1}^R u_{\tau_r,\gamma_r}\cdot
u_{\bt_r,\gamma_r}.
$$
We can define $c_+=c_{0,+}c_{1,+}$ and obtain the desired factorization $c=t^{2n}c_+\tc_+\iv$.
Because of the conditions (i*)--(iii*) it is easy to check that $c_{1,+}$ and its inverse satisfy
conditions as in (\ref{f.f2}), but with $p$ and $q$ replaced by $\hat{p}$ and $\hat{q}$.
Now multiplication with $c_{0,+}$ yields that $c_+$ and its inverse satisfy the conditions (\ref{f.f2}).
This completes the proof.
\end{proof}

The previous result is important in as far as it shows that under the assumption of Fredholmness 
a factorization exists. In the next section we will compute the defect numbers of $T(a)+H(b)$
under the assumption of Fredholmness and that $c,d$ admit a factorization.  Thus the Fredholmness assumption is sufficient for our formulas for the defect numbers (in the piecewise continuous case).

Considering $L^\iy$-symbols one can ask the question whether the Fredholmness of
$T(a)+H(b)$ implies the existence of an anti-symmetric factorization of the auxiliary functions.
That this is not the case shows the limitations for our formulas for the defect numbers.

\begin{example}\em
Let $a=1$ and $b(t)=\exp(\frac{t-1}{t+1})$. Since $b\in H^\iy(\T)$ and $b\iv=\tilde{b}\in \overline{H^\iy(\T)}$
we have that $T(a)+H(b)=I$ is invertible, while on the other hand the functions
$c=d=b\iv$ do not admit factorizations of the above kind as can be shown easily.
\qed
\end{example}

%%%%%%%%%%%%%%%%%%%%%%%%%%%%%%%%%%%%%%%%%%%%%%%%
%%%%%%%%%%%%%%%%%%%%%%%%%%%%%%%%%%%%%%%%%%%%%%%%

\section{Computation of the defect numbers}
\label{s4}

In what follows, let $1<p<\iy$ and $1/p+1/q=1$. We say that 
$a,b,c,d$ satisfy the {\em basic factorization conditions in $H^p(\T)$} if the following holds:
\begin{itemize}
\item[(i)]
$a,b\in G L^\iy(\T)$ satisfy $a\ta=b\tb$ and 
$$
c=\frac{a}{b}=\frac{\tb}{\ta},\qquad d=\frac{\ta}{b}=\frac{\tb}{a}.
$$
\item[(ii)]
$c$ and $d$ admit factorizations of the form
$$
c(t)=c_+(t) t^{2n} c_+\iv(t\iv),
\qquad
d(t)=d_+(t) t^{2m} d_+\iv(t\iv),
$$
with $n,m\in\Z$  and
$$
(1+t)c_+(t) \in \Hq,\qquad  (1-t) c_+\iv(t)\in \Hp,
$$
$$
(1+t) d_+(t)\in \Hp,\qquad
(1-t)d_+\iv(t) \in \Hq.
$$
\end{itemize}
We remark that the indices $n,m$ are uniquely determined and that functions $c_+$ and $d_+$ are also unique up to a multiplicative constant. This fact is easy to prove and follows for instance from Propositions 3.1 and 3.2 in \cite{BE2}. We will not make use of the uniqueness statement.

In the main theorem of this section we are going to compute the defect numbers of $T(a)+H(b)$
under the above conditions.
In this theorem we have to distinguish several cases, and for the proof of the probably most difficult
case we need the following well-known basic lemma.

\begin{lemma}\label{l4.1}
Let $X$ and $Y$ be Banach spaces, let $A:X\to Y$ be a linear bounded and invertible operator,  let
$P_1$ be a projection operator on $X$, and let $P_2$ be a projection operator on $Y$. Put $Q_1=I-P_1$ and $Q_2=I-P_2$.
Then the Banach spaces $\ker P_2 A P_1$ and $\ker Q_1 A\iv Q_2$ are isomorphic.
\end{lemma}
\begin{proof}
Let $R$ denote the restriction of the operator $P_1A\iv Q_2$ onto $\im Q_2$ mapping into 
$\im P_1$. For $x\in \ker Q_1A\iv Q_2$ we compute
$$
(P_2AP_1)Rx=P_2AP_1A\iv Q_2 x=P_2AA\iv Q_2 x-P_2AQ_1A\iv Q_2x=0.
$$
Hence $R$ maps $\ker Q_1 A\iv Q_2$ into $\ker P_2A P_1$. Similarly, the operator 
$S$ defined as the restriction of $Q_2AP_1$ onto $\im P_1$ maps 
$\ker P_2A P_1$ into  $\ker Q_1 A\iv Q_2$. Now we observe that for $x\in \ker P_2 A P_1$ 
we have
$$
RSx=P_1 A\iv Q_2 A P_1x=P_1-P_1 A\iv  P_2 A P_1 x=P_1 x.
$$
Hence $RS$ equals the identity operator on $\ker P_2 A P_1$. Similarly, $SR$ equals the identity 
operator on $\ker Q_1 A\iv Q_2$. {}From this the desired assertion follows.
\end{proof}

\begin{theorem}\label{t3.2}
Assume that $a,b,c,d$ satisfy the basic factorization conditions in $H^p(\T)$ with $n,m\in\Z$. 
Suppose that $T(a)+H(b)$ is Fredholm on $\Hp$. Then 
\bq\label{f.48}
\dim\ker( T(a)+H(b)) &=&\left\{
\begin{array}{llr} 0 & \mbox{ if }n > 0, \ m\le 0 & \mathrm{(G)} \\
-n & \mbox{ if } n\le 0, \ m\le 0& \mathrm{(G)}\\
\dim \ker A_{n,m} \hspace*{2.6ex}& \mbox{ if } n>0 , \ m>0& \mathrm{(F)}\\
m -n & \mbox { if } n\le 0,\ m>0,& \mathrm{(F)}
\end{array}\right.
\\[3ex]
\label{f.49}
\dim\ker (T(a)+H(b))^* &=& \left\{
\begin{array}{llr}  0 & \mbox{ if }m >  0, \ n\le 0& \mathrm{(G)} \\
 -m & \mbox{ if } m\le  0, \ n\le 0& \mathrm{(G)}\\
 \dim \ker (A_{n,m})^T & \mbox{ if } m>0, \ n>0& \mathrm{(F)}\\
 n-m & \mbox { if } m\le 0,\ n>0.& \mathrm{(F)}
\end{array}\right.
\eq
Therein, in case $n>0$, $m>0$,
$$
 A_{n,m}:=
\left[ \rho_{i-j}+\rho_{i+j}\right]_{i=0}^{n-1}\ _{j=0}^{m-1}.
$$
and 
$$
\rho:=t^{-m-n} (1+t)(1+t\iv) \tc_+\td_+ b\iv \in L^1(\T).
$$
In particular, the Fredholm index of $T(a)+H(b)$ is equal to $m-n$.
\end{theorem}

Before we begin with the proof, let us remark that the cases marked with (G) do not require the assumption that $T(a)+H(b)$ is Fredholm. The cases marked with (F) seem to require the assumption that $T(a)+H(b)$ is Fredholm. Since we are not going to present counterexamples here, the question whether Fredholmness is actually needed may be deferred to future investigations. Without the Fredholmness we will prove that the right hand sides are an upper estimate for the dimensions of the respective kernels. 

\begin{proof}
In what follows, let $\cP$ stand for the set of all polynomials in $t$ and let 
$$
\cP^\pm_\kappa:=\Big\{ p\in \cP\;:\; \pm  t^{\kappa} \tilde{p}=p\Big\}.
$$
It is easy to see that
$$
\cP_{2\kappa}^-=(1-t^2)\cP_{2\kappa-2}^+,\quad
\cP_{2\kappa+1}^+=(1+t)\cP_{2\kappa}^+,\quad
\cP_{2\kappa+1}^-=(1-t)\cP_{2\kappa}^+.
$$
and $\dim P^+_{2\kappa}=\kappa+1$ in case $\kappa\ge 0$ and $=0$ in case $\kappa<0$.
Moreover, we introduce
$$
\cP_{n,m}=\Big\{ p\in L^1(\T)\,:\, p_k=0 \mbox{ for } k<n \mbox{ or } k>m \Big\}.
$$ 
We will use this notation and these basic facts in what follows.

Let us start with analyzing the dimension of the kernel of $T(a)+H(b)$.
We need to find the dimension of the space of all $f_+\in \Hp$ such that 
$(T(a)+H(b))f_+=0$. For each such $f_+$ there exist a (unique) $f_-\in t\iv \ovl{\Hp}$ such that 
$$
a f_+  +   t^{-1} b \tf_+= f_-.
$$
Multiplying with $(1-t)\td_+\iv$  we obtain
$$
f_0:=(1-t)\td_+\iv a f_+ -(1-t\iv)\td_+\iv b \tf_+ = -t (1-t\iv)\td_+\iv f_-.
$$
The right hand side is in $\ovl{H^1(\T)}$ while  the left hand side $f_0$ satisfies the relation
$\tf_0=-t^{2m}f_0$ because
$$
d_+\iv \ta=d_+\iv d b=t^{2m} \td_+\iv b.
$$
Hence $f_0=\tilde{p}_1$ with $p_1\in \cP_{2m}^{-}$. Thus we can write
$p_1=(1-t^2) p_2$ with $p_2\in \cP^+_{2m-2}$. Now we can eliminate $f_-$ and obtain
$$
f_-=-t\iv (1-t\iv)^{-1} \td_+ \tilde{p}_1=-t\iv(1+t\iv)\td_+ \tilde{p}_2.
$$
Since for each polynomial $p_2$, the function $f_-$ defined in this way belongs to $t\iv \ovl{\Hp}$, we are left with determining the dimension of the space of all $(f_+,p_2)\in \Hp\times\cP^{+}_{2m-2}$
for which
$$
(1-t)\td_+\iv  a f_+ -(1-t\iv)\td_+\iv  b \tf_+= (1-t^{-2})\tilde{p}_2.
$$

Multiplying with 
\be\label{f.fm}
t^{n} c_+\td_+ a\iv \frac{1+t}{1-t}=
-t^{-n} \tc_+\td_+ b\iv \frac{1+t\iv}{1-t\iv} =
- \frac{t^{m-1}    \rho}{1-t^{-2}}
\ee
we obtain
$$
(1+t)t^{n}c_+ f_+ + (1+t\iv) t^{-n} \tc_+ \tf_+ = 
- t^{m-1}\rho\tilde{p}_2.
$$
Each of terms on the left hand side is in $L^1(\T)$. So is the right hand side because
$\rho\in L^1(\T)$. We substitute $p_3=t^{-m+1}p_2=t^{m-1}\tilde{p}_2$
and obtain
\bq\label{f.e2}
(1+t)t^{n}c_+ f_+ + (1+t\iv) t^{-n} \tc_+ \tf_+ = 
-\rho p_3
\eq
with $p_3$ subject to the condition $p_3\in \cP_{-m+1,m-1}$ and $\tilde{p}_3=p_3$.
Since the argumentation can be reversed, we are led to determine the space of all 
$(f_+,p_3)$ satisfying (\ref{f.e2}) and the mentioned conditions.

We now have to distinguish several cases.

\underline{Case $n> 0$, $m\le 0$.}\
Because $m\le0$ we have $p_3\equiv0$, and hence we need to look at the homogeneous version of (\ref{f.e2}). Noting that $(1+t) c_+ f_+\in H^1(\T)$, the condition $n\ge 0$ easily implies (by inspecting the
Fourier coefficients) that $(1+t) c_+ f_+\equiv0$, whence $f_+\equiv 0$.
Thus we conclude that the kernel of $T(a)+H(b)$ is trivial.

\underline{Case $n\le 0$, $m\le 0$.}\
As before we are led to the  homogeneous version of (\ref{f.e2}). If we define
$q_1:=(1+t)c_+ f_+\in H^1(\T)$, the condition is that $q_1=-t^{-2n}\tilde{q}_1$.
This implies that $q_1$ is a polynomial of degree at most $-2n$ and moreover that 
$q_1\in \cP^-_{-2n}$. We can write $q_1=(1-t^2)q_2$
with $q_2\in \cP^+_{-2n-2}$. Substituting back we obtain
$f_+=(1-t)c_+\iv q_2$ which always belongs to $\Hp$. From this we can conclude that 
the dimension of the kernel of $T(a)+H(b)$ equals the dimension of  $\cP^+_{-2n-2}$,
which is $-n$.

\underline{Case $n>0$, $m>0$.}\ 
Here, for $-n<k<n$, the $k$-th Fourier coefficient on the right hand side  
of (\ref{f.e2}) has to vanish:
\begin{equation}\label{e.fp3}
\left[\rho p_3\right]_{k}=0, \qquad -n<k<n.
\end{equation}
This finite Toeplitz system we are going to transform (equivalently) further by making use of $\tilde{\rho}=\rho$ and the requirements on $p_3$. (The fact that $\rho\in L^1(\T)$ and that $\rho=\tilde{\rho}$ can be checked easily.) Thus we can write $p_3$ uniquely in the form
$p_3=p_4+\tilde{p_4}$ with $p_4\in \cP_{0,m-1}$. Moreover, this correspondence between
$p_3$ and $p_4$ is one-to-one. We arrive at the equivalent condition
\be\label{e.thr}
\sum_{i=0}^{n-1}\left(\rho_{k-i}p_{4,i}+\rho_{k+i}p_{4,i}\right) =0,\qquad -n<k<n.
\ee
A moments though shows that the conditions for $-n<k<0$ are the same as for
$0<k<n$. Thus we can restrict to $0\le k<n$. Then this system corresponds precisely to the
matrix $A_{m,n}$. We conclude that the dimension of the space of solutions $p_3$ satisfying 
(\ref{e.fp3}) equals the dimension of the kernel of $A_{m,n}$. At this point we have shown that 
for each solution $(f_+,p_3)$ of (\ref{e.fp3}) we can assign a polynomial $p_4$ satisfying (\ref{e.thr}). This mapping is linear and injective. Indeed, the injectiveness follows from the assumption $n>0$. We can conclude that 
$$
\dim\ker (T(a)+H(b))\le \dim \ker A_{n,m},
$$
which confirms the statement made after the theorem. So far we have not used the Fredholmness assumption. The proof of equality in the Fredholm case will be given later.

\underline{Case $n\le 0$, $m>0$.}\
Here we consider the (linear) mapping $\Lambda$ which sends the space of all solutions
$(f_+,p_3)$ of  (\ref{f.e2}) to $p_3$. Because $p_3\in \cP_{-m+1,m-1}$ and $\tilde{p}_3=p_3$, the
dimension of the range of this mapping is at most $m$. On the other, as follows from the argumentation given in the second case ($n\le 0$, $m\ge 0$), the kernel of $\Lambda$ is equal to $-n$. Hence we conclude that the space of solutions $(f_+,p_3)$ is at most $m-n$, i.e., 
$$
\dim\ker (T(a)+H(b))\le m-n.
$$
Again, we have not used the assumption about the Fredholmness, and we will complete the discussion of this case later.

Let us now consider the kernel of the adjoint. If we identify the dual of $\Hp$ with $\Hq$ via the mapping
$g\in \Hq\mapsto \langle g,\cdot\rangle\in (\Hp)'$ where
$$
\langle g,f \rangle =\frac{1}{2\pi} \int_{0}^{2\pi} g(e^{-ix}) f(e^{i\theta})\,d\theta,
$$
then the adjoint of $T(a)+H(b)$ is identified with  $T(\tilde{a})+H(b)$. To give another interpretation,
if we consider the matrix representations with respect to the standard bases $\{t^n\}_{n=0}^\iy$ in $\Hp$ and $\Hq$, then $T(\tilde{a})+H(b)$ is the transpose of $T(a)+H(b)$. In any case,
$$
\dim \ker (T(a)+H(b))^*=\dim\ker( T(\ta)+H(b)).
$$
Now we could either reproduce the above arguments and prove formula (\ref{f.49}) in the cases (G) and 
the estimates ``$\le$'' in the cases (F). We can also derive this in a simple manner, by noting that the passage of the operator $T(a)+H(b)$ on $\Hp$ to the operator $T(\tilde{a})+H(b)$ on $\Hq$ corresponds to the following changes:
$$
p\leftrightarrow q,\quad
a\leftrightarrow \tilde{a},  \quad b \leftrightarrow b,  \quad c \leftrightarrow d, \quad 
c_+ \leftrightarrow d_+, \quad
n\leftrightarrow m, \quad \rho\leftrightarrow\rho
$$
This is, in particular, compatible with the assumption that $a,b,c,d$ satisfy the basic factorization conditions. Let us record the corresponding inequalities in the cases (F) of (\ref{f.49}). In case $m>0$, $n>0$ we can conclude that
$$
\dim\ker (T(a)+H(b))^*\le \dim\ker A_{m,n}.
$$
Notice that $A_{m,n}=A_{n,m}^T$. In case $m\le0$, $n>0$ we can establish that 
$$
\dim\ker (T(a)+H(b))^*\le n-m.
$$

Next let us determine the Fredholm index. We assume that $T(a)+H(b)$ is Fredholm.
Given the factorization of the corresponding functions $c,d$ with $n,m$ we define
$$
a'(t)=a(t) t^{m-n}, \qquad b'(t)=b(t)t^{n+m}.
$$
Then the corresponding functions $c'$ and $d'$ evaluate to 
$$
c'(t)=c(t) t^{-2n},\qquad d'(t)=d(t)t^{-2m}.
$$
It follows that $a',b',c',d'$ satisfies the basic factorization conditions with 
$n=m=0$ (and the same $c_+,d_+$). Because for $n=m=0$ the formulas (\ref{f.48}) and (\ref{f.49}) have been proved, and we are able to conclude that the kernel of both $T(a')+H(b')$ and its adjoint are trivial.
To proceed, we observe that 
$$
T(a)+H(b) = T(a't^{n-m})+H(b't^{-n-m})
$$
This equals, modulo compact operators,
$$
T(t^{-m})T(a')T(t^n)+T(t^{-m})H(b') T(t^n)
$$
by using (\ref{Tab}) and (\ref{Hab}).
Because $T(t^{-m})$ and $T(t^n)$ are Fredholm (with index $m$ and $-n$) we can now conclude that 
$T(a)+H(b)$ is Fredholm if and only of $T(a')+H(b')$ is Fredholm. Since $T(a)+H(b)$ is assumed to be Fredholm it follows that $T(a')+H(b')$ is Fredholm. In fact, it is invertible due to the vanishing of the defect numbers. Now we can conclude that the Fredholm index of $T(a)+H(b)$ is $m-n$.

Now let us return to the outstanding cases (F) of formulas (\ref{f.48}) and (\ref{f.49}).
First consider the case $n\le 0$, $m>0$ of (\ref{f.48}). The inequality has already been proved.
The equality follows from the Fredholm index formula and from the respective case of  formula (\ref{f.49})  (which again has already been proved).
Similarly, we can argue in the case $m\le0$, $n>0$ of formula (\ref{f.49}).

Unfortunately, the case $n,m>0$ of both formulas cannot be settled by such an argument.
Introduce $a'$ and $b'$ as above. Then we have for the Fourier coefficients
$$
[a']_k=a_{k-m+n},\quad [b']_k=b_{k-n-m},
$$
whence $[a']_{i-j}+b'_{1+i+j}=a_{(i-m)-(j-n)}+b_{(i-m)+(j-n)+1}$. Considering the matrix representations with respect to the standard bases, we observe that we obtain
$A=T(a)+H(b)$ by deleting the first $m$ rows and first $n$ columns of $A'=T(a')+H(b')$.
Introduce the projection operator
$$
P_n:\sum_{k=0}^\iy f_k e^{ik\theta}\in \Hp\mapsto \sum_{k=0}^{n-1} f_k e^{ik\theta}\in \Hp
$$
and its complementary projection $Q_n=I-P_n$. Then, what we have just said implies that 
$A\cong Q_m A' Q_n$, where we consider last operator as acting from the image of $Q_n$
into the image of $Q_m$. Now we use the fact that $A'$ is invertible  on $\Hp$, 
which has been shown above, and apply Lemma \ref{l4.1} in order to conclude that 
$$
\dim \ker A =\dim \ker (P_n (A')\iv P_m).
$$
In order to prove the desired assertion, we are going to establish the formula
$$
T_n(c_+(1+t)) (P_n (A')\iv P_m) T_m(\tilde{d}_+(1+t\iv))=D_n A_{n,m} D_m
$$
where $T_n(\phi):= P_n T(\phi) P_n$ is a finite Toeplitz matrix and 
$D_n:=\mathrm{diag}(1/2,1,1,\dots,1)$ is an $n\times n$ diagonal matrix.
Clearly, we identify operators on the space $\Hp$ with their matrix representation with respect to the 
standard basis.
Before we start we remark that the matrices $T_n(c_+(1+t))$ and $T_m(\tilde{d}_+(1+t\iv))$
are lower and upper triangular Toeplitz matrices with non-zero entries on the main diagonal.

We are going to prove the identity by computing the $(j,k)$-entry, $0\le j<n$, $0\le k<m$.
For this apply the matrix on the left to the function $t^k$. Put
$$
g_+=T_m(\tilde{d}_+(1+t\iv)) t^k.
$$
Clearly $g_+$ is a polynomial in $t$, which is determined by
$$
g_+=\tilde{d}_+(1+t\iv) t^k+t\iv \ovl{\Hp}.
$$
Next we put $f_+=(A')\iv g_+\in \Hp$. Equivalently, $g_+=A' f_+$, or
$$
g_+=a' f_++b' t\iv \tilde{f}_++t\iv \ovl{\Hp}.
$$
Now combine both equations to conclude that
$$
\tilde{d}_+(1+t\iv) t^k=a ' f_++b't\iv \tilde{f}_++t\iv \ovl{\Hp}.
$$
We proceed by using the same ideas as at the very beginning of the proof. We multiply with 
$(1-t)\tilde{d}_+\iv $ to get
$$
(t\iv - t)t^k = f_0+ \ovl{H^1(\T)}, \quad f_0:= (1-t)\tilde{d}_+\iv a' f_+ -(1-t\iv )\tilde{d}_+\iv b' \tilde{f}_+ 
$$
where $f_0=-\tilde{f}_0$. In case $k=0$ we leave the equation as it is and conclude that
$t\iv-t=f_0$ by inspecting the Fourier coefficients. In case $k\ge 1$ we substract
$(t-t\iv)t^{-k}\in \ovl{H^1(\T)}$ and obtain compare the Fourier coefficients. In both cases we 
conclude that 
$$
\sigma_k(t\iv-t)(t^k+t^{-k})=f_0
$$
where $\sigma_0=1/2$ and $\sigma_k=1$ for $k\ge1$. Now we multiply with the respective 
expression (\ref{f.fm}), noting that $n=m=0$, and obtain
$$
\rho \sigma_k(t^k+t^{-k})=(1+t)c_+f_++(1+t\iv) \tilde{c}_+ \tilde{f}_+.
$$
All terms belong to $L^1(\T)$ and we obtain for the $j$-th Fourier coefficient
$$
\sigma_j\sigma_k\Big[\rho(t^k+t^{-k})\Big]_j=\Big[(1+t)c_+ f_+\Big]_j.
$$
Now it is easy to conclude that the left hand side is the $(j,k)-entry$ of the matrix
$D_n A_{n,m} D_m$ and that the right hand side is the $j$-th entry of
the vector $T_n((1+t)c_+) P_nf_+$. This proves the desired identity and concludes the proof of
the case $n,m>0$ of formula (\ref{f.48}). The corresponding case for formula (\ref{f.49}) can be derived
similarly, or by making use of the Fredholm index of $T(a)+H(b)$.
\end{proof}

We can now state a simple conclusion concerning invertibility.

\begin{corollary}
Let $a,b,c,d$ satisfy the basic factorization conditions for $\Hp$. Suppose that 
$T(a)+H(b)$ is Fredholm on $\Hp$. Then $T(a)+H(b)$ is invertible on $\Hp$ if and only if
\begin{enumerate}
\item[(i)]
$n=m=0$, or,
\item[(ii)]
$n=m>0$ and $A_{n,n}$ is invertible.
\end{enumerate}
\end{corollary}

%%%%%%%%%%%%%%%%%%%%%%%%%%%%%
%%%%%%%%%%%%%%%%%%%%%%%%%%%%%

\section{Some special cases}
\label{s5}

Throughout this section, let $PC_0$ stand for the set of all piecewise continuous functions with 
finitely many jump discontinuities. We denote by $GPC_0$ the set of all invertible function 
in $PC_0$. Moreover, we let $G_0 C(\T)$ stand for the set of all nonvanishing continuous functions on $\T$ with winding number zero.
Notice that $a\in G_0C(\T)$ if and only if $a$ has a continuous logarithm.
 Again, throughout the section, $1<p<\iy$ and $1/p+1/q=1$.

We will restrict to functions from $PC_0$ only to simplify the presentation of the results.

\subsection{The case of operators $\boldsymbol{T(a)\pm H(a)}$}

The case of $T(a)+H(a)$ has already been established in \cite[Thm.~7.4]{BE2}. We derive it here again by employing our previous results.

\begin{theorem} 
Let $a\in GPC_0$. Then $A=T(a)+H(a)$ is Fredholm  on $\Hp$ if and only if $a$ can be represented
in the form
\be\label{f.p1}
a(t) =  t^{\kappa} a_0(t) u_{1,\beta^+}(t) u_{-1,\beta^-}(t)\prod_{r=1}^R
u_{\tau_r,\beta_r^+}(t) u_{\bar{\tau}_r,\beta_r^-}(t) 
\ee
with $\tau_1,\dots\tau_R\in\T_+$ being distinct, $\kappa\in\Z$, $a_0\in  G_0C(\T)$, and
\begin{enumerate}
\item[(i)] $-1/2-1/2q<\Re \beta^+<1/2p$, 
\item[(ii)] $-1/2q<\Re \beta^{-}< 1/2+1/2p$, 
\item[(iii)]  $-1/q<\Re(\beta_r^++\beta_r^-) <1/p$ for $1\le r\le R$.
\end{enumerate}
Moreover, in this case
$$
\ker A=\left\{\ba{cl}
0 & \mbox{ if }\kappa\ge0 \\ -\kappa & \mbox{ if }\kappa<0, \ea\right.\qquad
\ker A^*=\left\{\ba{cl}
\kappa & \mbox{ if } \kappa>0\\ 0 &  \mbox{ if }\kappa\le 0. \ea\right.
$$
\end{theorem}
\begin{proof}
Because $a\in GPC_0$ we can always represent it in the form (\ref{f.p1}) 
with perhaps not yet properly chosen parameters $\beta$ and $\varkappa$. Notice that we can always
add an integers to the $\beta$ parameters at the expanse of modifying $\kappa$ and changing $a_0$
by a constant. In order to apply the results of the previous sections, we evaluate the corresponding functions $c=1$ and $d=\tilde{a}/a$. There is nothing to check for $c$.
For $d$ we obtain
$$
d(t) = t^{-2\kappa} \ta_0(t) a_0\iv(t) u_{1,-2\beta^+}(t) u_{-1,-2\beta^-}(t) \prod_{r=1}^R
u_{\tau_r,-\beta_r^+-\beta_r^-}(t) u_{\bar{\tau}_r,-\beta_r^+-\beta_r^-}(t).
$$
If (i)--(iii) are satisfied, then $A$ is Fredholm by Corollary \ref{c2.2} (see also (\ref{u.pm})).
Conversely, if $A$ is Fredholm, then (\ref{f.c10})--(\ref{f.c12}) hold and we can modify
the parameters $\beta$ and $\kappa$ by adding integers such that (i)--(iii) hold and
the winding number of $a_0$ is zero. Notice that we also achieve a representation (\ref{f.14b})
as in Proposition \ref{p2.4}, where
$$
\delta^+=-\beta^+,\quad \delta^-=-\beta^-,\quad \delta_r=-\beta_r^+-\beta_r^-.
$$
Furthermore, by Theorem \ref{t3.3x} we have a factorization of $d$ as (\ref{t3.3x}) with $m=-\kappa$.
The function $c\equiv1$ has a trivial factorization with $c_+\equiv1$ and $n=0$.
It remains to apply Theorem \ref{t3.2} to identify the defect numbers.
\end{proof}

The case of $T(a)-H(a)$ is quite analogous. Notice that, since
$T(a)-H(a)=\widehat{W}(T(\hat{a})+H(\hat{a}))\widehat{W}$ with $(\widehat{W}f)(t)=f(-t)$ and $\hat{a}(t)=a(-t)$,
it can be reduced to the former case. It simply amounts to interchange the conditions  (i) and (ii).

%%%%%%%%%%%%%%%%%%%%%%%%

\subsection{The cases of operators $\boldsymbol{T(a)-H(t\iv a)}$ and $\boldsymbol{T(a)+H(t a)}$}

The cases of operators
$$
T(a)-H(t\iv a)\cong(a_{j-k}-a_{j+k+2})_{j,k=0}^\iy,\qquad
T(a)+H(t\iv a)\cong(a_{j-k}+a_{j+k})_{j,k=0}^\iy.
$$
also admit fairly simple characterizations. 
They have been examined recently in the case of smooth symbols in \cite{BE4}.

\begin{theorem} 
Let $a\in GPC_0$. Then $A=T(a)-H(t\iv a)$ is Fredholm  on $\Hp$ if and only if $a$ can be represented
in the form
\be\label{f.p2}
a(t) =  t^{\kappa} a_0(t) u_{1,\beta^+}(t) u_{-1,\beta^-}(t)\prod_{r=1}^R
u_{\tau_r,\beta_r^+}(t) u_{\bar{\tau}_r,\beta_r^-}(t) 
\ee
with $\tau_1,\dots\tau_R\in\T_+$ being distinct, $\kappa\in\Z$, $a_0\in  G_0C(\T)$, and
\begin{enumerate}
\item[(i)] $-1/2q<\Re \beta^+<1/2+1/2p$, 
\item[(ii)] $-1/2q<\Re \beta^{-}< 1/2+1/2p$, 
\item[(iii)]  $-1/q<\Re(\beta_r^++\beta_r^-) <1/p$ for $1\le r\le R$.
\end{enumerate}
Moreover, in this case
$$
\ker A=\left\{\ba{cl}
0 & \mbox{ if }\kappa\ge0 \\ -\kappa & \mbox{ if }\kappa<0, \ea\right.\qquad
\ker A^*=\left\{\ba{cl}
\kappa & \mbox{ if } \kappa>0\\ 0 &  \mbox{ if }\kappa\le 0. \ea\right.
$$
\end{theorem}
\begin{proof}
The proof of this theorem goes essentially in the same way as in the previous theorem.
We remark only the main computational steps. With $b=-t\iv a$ we obtain $c=a/b=-t$ and
$d=\ta/b=-t \ta/a$. We can write $c$ as $c=\eta_{1,1}\xi_{1,-1}=u_{1,1}$, whence $\gamma^+=1/2$ and $n=0$.
For $d$ we obtain
$$
d(t) = t^{-2\kappa} a_0(t) \tilde{a}_0\iv(t) u_{1,-2\beta^++1}(t) u_{-1,-2\beta^-}(t) \prod_{r=1}^R
u_{\tau_r,-\beta_r^+-\beta_r^-}(t) u_{\bar{\tau}_r,-\beta_r^+-\beta_r^-}(t).
$$
The only difference in comparison to the previous theorem is that $\delta^+=-\beta^++1/2$, which makes a change in condition (i). As before $m=-\kappa$.
\end{proof}

\begin{theorem} 
Let $a\in GPC_0$. Then $A=T(a)+H(t a)$ is Fredholm  on $\Hp$ if and only if $a$ can be represented
in the form
\be\label{f.p3}
a(t) =  t^{\kappa} a_0(t) u_{1,\beta^+}(t) u_{-1,\beta^-}(t)\prod_{r=1}^R
u_{\tau_r,\beta_r^+}(t) u_{\bar{\tau}_r,\beta_r^-}(t) 
\ee
with $\tau_1,\dots\tau_R\in\T_+$ being distinct, $\kappa\in\Z$, $a_0\in  G_0C(\T)$, and
\begin{enumerate}
\item[(i)] $-1/2-1/2q<\Re \beta^+<1/2p$, 
\item[(ii)] $-1/2-1/2q<\Re \beta^{-}< 1/2p$, 
\item[(iii)]  $-1/q<\Re(\beta_r^++\beta_r^-) <1/p$ for $1\le r\le R$.
\end{enumerate}
Moreover, in this case
$$
\ker A=\left\{\ba{cl}
0 & \mbox{ if }\kappa\ge0 \\ -\kappa & \mbox{ if }\kappa<0, \ea\right.\qquad
\ker A^*=\left\{\ba{cl}
\kappa & \mbox{ if } \kappa>0\\ 0 &  \mbox{ if }\kappa\le 0. \ea\right.
$$
\end{theorem}
\begin{proof}
Here we have $b=t a$, and we obtain $c=a/b=t$ and
$d=\ta/b=t\iv \ta/a$. We can write $c$ as $c=\eta_{-1,-1}\xi_{-1,1}=u_{-1,-1}$, whence $\gamma^-=-1/2$ and $n=0$.
For $d$ we obtain
$$
d(t) = t^{-2\kappa} a_0(t) \tilde{a}_0\iv(t) u_{1,-2\beta^+}(t) u_{-1,-2\beta^--1}(t) \prod_{r=1}^R
u_{\tau_r,-\beta_r^+-\beta_r^-}(t) u_{\bar{\tau}_r,-\beta_r^+-\beta_r^-}(t).
$$
Now $\delta^+=-\beta^+-1/2$, which causes a change in  (ii).
\end{proof}

%%%%%%%%%%%%%%%%%%%%%%%%

\subsection{The case of operators $\boldsymbol{I+H(\phi)}$}

We want to derive Fredholm criteria and formulas for the index in the case of operators
$I+H(\phi)$. In order to apply our assumptions we must assume that $\phi\tilde{\phi}=1$, which derives from the condition $a\ta=b\tb$.

For notational convenience we will consider the operator $I+H(\tilde{\phi})$ instead of
$I+H(\phi)$, i.e., we will replace $\phi$ by $\tilde{\phi}=\phi\iv$. Moreover, for simplicity we will assume that $\phi\in GPC_0$. In this case, it is possible to write $\phi$ in the form
$$
\phi=a_1\cdot u_{1,\beta^+}\cdot u_{-1,\beta^-}\prod_{r=1}^R u_{\tau_r,\beta_r}\cdot u_{\bt_r,\beta_r}.
$$
Here $a_1$ is continuous and nonvanishing on $\T$, and $a_1\ta_1=1$. Notice that the size of the jump at $\tau_r$ and $\bt_r$ is the same. This representation is not unique. We can replace $\beta_r$ by $\beta_r+n_r$. This leads to an additional factor $t^{2n}$, which we can combine with $a_1$.
Next we would like to guarantee that $a_1(\pm1)=1$. However, since $a_1\ta_1=1$,  we only have $a_1(1)=\pm1$ and $a_1(-1)=\pm1$.
In order to correct this we multiply $a_1$ with either $u_{1,1}(t)=-t$ or $u_{-1,1}(t)=t$ or both, and
decrease the value of $\beta^+$ or $\beta^-$ by one (see also (\ref{u.p})). In this way, we achieve the conditions $a_1(\pm 1)=1$. Now we can modify $\beta^+$ and $\beta^-$ only by integer multiples of $2$.
Since  the continuous function $a_1$ satisfies $a_1\ta_1=1$ and $a_1(\pm1)=1$, it is easily seen that $a_1$ has an even winding number. Hence we can write $a_1=t^{2\kappa}a_0$. This leads us to the following representation
\bq\label{f.39}
\phi=t^{2\kappa} \cdot a_0\cdot u_{1,\beta^+}\cdot u_{-1,\beta^-}\prod_{r=1}^R u_{\tau_r,\beta_r}\cdot u_{\bt_r,\beta_r}.
\eq
where $a_0\in G_0C(\T)$ satisfies $a_0\ta_0=1$, $a_0(\pm1)=1$.
To conclude, we remark that the representation is not unique, but that we can make the following replacements,
\be\label{f.40}
\beta^{\pm}\mapsto\beta^{\pm}+2n^{\pm},\quad \beta_r\mapsto \beta_r+n_r,\quad
\kappa\mapsto 
\kappa-n^+-n^--n_1-\dots- n_R,
\ee
while $a_0$ stays the same. The above conditions on $a_0$ mean that $a_0\in \cB_0$ (in the notation of Sect.~\ref{s3}), and due to Lemma \ref{l.1} (and the remark made afterwards)
there exists a factorization
\bq\label{a0-f}
a_0= a_{0,+}\cdot (\ta_{0,+})\iv
\eq
with $a_{0,+},a_{0,+}\iv\in \bigcap\limits_{\sigma>0}H^{\sigma}(\T)$.

In order to analyze the Fredholmness of $I+H(\tilde{\phi})$, we note that $a=1$ and $b=\phi\iv$. Hence the corresponding auxiliary functions evaluate to  $c=d=\phi$. The connection between the parameters is $\gamma_r=\delta_r=\beta_r$, $\gamma^+=\delta^+=\beta^+/2$, and
$\gamma^-=\delta^-=\beta^-/2$ modulo the integers. Using Corollary \ref{c2.2} we arrive at the conclusion that $I+H(\tilde{\phi})$ is Fredholm on $\Hp$ if and only if there exist a representation (\ref{f.40}) such that the following
conditions are fulfilled:
\begin{enumerate}
\item[(i)]
$\Re\beta^+\notin\left\{-\frac{1}{p},-\frac{1}{q}\right\}+ 2\,\Z$,
\item[(ii)]
$\Re\beta^-\notin\left\{\frac{1}{p},\frac{1}{q}\right\}+2\,\Z$,
\item[(iii)]
$\Re\beta_r\notin\left\{\frac{1}{p},\frac{1}{q}\right\}+\Z$ for each $1\le r \le R$.
\end{enumerate}
Notice that the conditions arise as combinations between the $\gamma$ and $\delta$ parameter conditions (see also Proposition \ref{p2.4}). 

Assuming that (i)--(iii) are satisfied, in the case $p=2$ we have ``canonical'' choices for the $\beta$ parameters by stipulating that $\beta^\pm$ belong to an interval of length 2 and the $\beta_r$ belong to an interval of length $1$. This canonical choice breaks down in the case $p\neq q$, where we have two choices (which may or may not coincide).

\begin{theorem}
Let $\phi\in GPC_0$ such that $\phi\tilde{\phi}=1$. Then 
$I+H(\tilde{\phi})$ is Fredholm on $\Hp$ if and only if $\phi$ can be written in the two forms 
\bq\label{f.r1}
\phi &=& t^{2n} \cdot a_{0}\cdot u_{1,2\gamma^+}\cdot u_{-1,2\gamma^-}\prod_{r=1}^R u_{\tau_r,\gamma_r}\cdot u_{\bt_r,\gamma_r}\\
\phi &=& t^{2m} \cdot a_{0}\cdot u_{1,2\delta^+}\cdot u_{-1,2\delta^+}\prod_{r=1}^R u_{\tau_r,\delta_r}\cdot u_{\bt_r,\delta_r}
\label{f.r2}
\eq
such that $n,m\in\Z$, $a_{0}\in G_0 C(\T)$,  $a_{0}\ta_{0}=1$, $a_{0}(\pm1)=1$, and
\begin{enumerate}
\item[(i)] $-1/2q<\Re\gamma^+<1/2+1/2p$, $-1/2p<\Re \delta^+<1/2+1/2q$, 
\item[(ii)] $-1/2-1/2q<\Re\gamma^-<1/2p$, $-1/2-1/2p<\Re \delta^-<1/2q$, 
\item[(iii)] $-1/q<\Re\gamma_r<1/p$, $-1/p<\Re \delta_r<1/q$, 
\end{enumerate}
In case of Fredholmness, the index equals $m-n$, and the defect numbers are 
\bq%\label{f.48}
\dim\ker(I+H(\tilde{\phi})) &=&\left\{
\begin{array}{ll} 0 & \mbox{ if }n > 0, \ m\le 0  \\
-n & \mbox{ if } n\le 0, \ m\le 0\\
\dim \ker A_{n,m} \hspace*{2.6ex}& \mbox{ if } n>0 , \ m>0\\
m -n & \mbox { if } n\le 0,\ m>0,
\end{array}\right.
\\[3ex]
%\label{f.49}
\dim\ker (I+H(\tilde{\phi}))^* &=& \left\{
\begin{array}{ll}  0 & \mbox{ if }m >  0, \ n\le 0\\
 -m & \mbox{ if } m\le  0, \ n\le 0\\
 \dim \ker (A_{n,m})^T & \mbox{ if } m>0, \ n>0\\
 n-m & \mbox { if } m\le 0,\ n>0.
\end{array}\right.
\eq
Therein, in case $n>0$, $m>0$, we have
$$
 A_{n,m}:=
\left[ \rho_{i-j}+\rho_{i+j}\right]_{i=0}^{n-1}\ _{j=0}^{m-1}.
$$
and 
\bq
\rho &=& t^{-m+n} a_{0,+}\ta_{0,+}
\eta_{1,2\gamma^+}\xi_{1,2\delta^+}\eta_{-1,2\gamma^-+1}\xi_{-1,2\delta^-+1}
\prod_{r=1}^R 
\eta_{\tau,\gamma_r}\xi_{\tau,\delta_r} \eta_{\bt,\gamma_r}\xi_{\bt,\delta_r}
\nn\\
&=&  t^{m-n} a_{0,+}\ta_{0,+}
\eta_{1,2\delta^+}\xi_{1,2\gamma^+}\eta_{-1,2\delta^-+1}\xi_{-1,2\gamma^-+1}
\prod_{r=1}^R 
\eta_{\tau,\delta_r}\xi_{\tau,\gamma_r} \eta_{\bt,\delta_r}\xi_{\bt,\gamma_r}\nn
\eq
where $a_0=a_{0,+}\cdot(\ta_{0,+})\iv$ is a factorization (\ref{a0-f}).
\end{theorem}
\begin{proof}
With the auxiliary functions $c=d=\phi\iv=\tilde{\phi}$ it is easy to see that the existence of the representations implies the Fredholmness. Conversely, we have already argued above that the Fredholmness implies two (in generel, different) representations with parameters which must be connected by (\ref{f.40}); thus the $a_0$ is the same in both representation. The representation (\ref{f.r1}) gives rise to the representation corresponding to $c$ and the representation (\ref{f.r2}) gives rise to the respresentation corresponding to $d$ (see also Proposition \ref{p2.4}). 

 This implies the index formula immediately. As for the defect numbers we rely on  Theorem \ref{t3.2} and Theorem \ref{t3.3x} for the factorization.
For $c=d=\phi$ we obtain the factorizations with $c_+$ and $d_+$ given by
\bq
c_{+} &=& a_{0,+}\cdot \eta_{1,2\gamma^+}\cdot \eta_{-1,2\gamma^-}\prod_{r=1}^R\eta_{\tau_r,\gamma_r}\cdot
\eta_{\bt_r,\gamma_r},
\nn\\
d_+ &=& a_{0,+}\cdot \eta_{1,2\delta^+}\cdot \eta_{-1,2\delta^-}\prod_{r=1}^R\eta_{\tau_r,\delta_r}\cdot
\eta_{\bt_r,\delta_r},
\nn
\eq
whence it follows that $\rho=t^{-m-n}(1+t)(1+t\iv)\tc_+\td_+b\iv$ with $b=\phi\iv=c\iv$. Hence
$$
\rho=t^{-m-n}(1+t)(1+t\iv)\tc_+\td_+c=t^{-m+n}(1+t)(1+t\iv)\td_+c_+\nn
$$
and
$$
\rho=t^{-m+n} a_{0,+}\ta_{0,+}
\eta_{1,2\gamma^+}\xi_{1,2\delta^+}\eta_{-1,2\gamma^-+1}\xi_{-1,2\delta^-+1}
\prod_{r=1}^R 
\eta_{\tau,\gamma_r}\xi_{\tau,\delta_r} \eta_{\bt,\gamma_r}\xi_{\bt,\delta_r}.
$$
For the second, equivalent, representation we exchange the roles of $c_+$ and $d_+$.
\end{proof}

A couple of remarks are in order to clarify the statement of the theorem.
Under the assumptions of the theorem,
$$
\gamma_r-\delta_{r} = \left\{
\ba{lc} 
0 \mbox{ or } 1 & \mbox{ if } p<2\\
 0 \mbox{ or } -1 & \mbox{ if } p>2\\
0 & \mbox{ if } p=2
\ea\right.
$$
and the same for $\gamma^\pm$ and $\delta^\pm$. This follows from (\ref{f.40}) and the ranges of the
$\gamma$ and $\delta$ parameters given in (i)--(iii).
Moreover, again by (\ref{f.40}),
$$
n-m=(\gamma^+-\delta^+) + (\gamma^--\delta^-) + \sum_{r=1}^R (\gamma_r-\delta_r). 
$$
The obvious conclusion is that the Fredholm index is always $\ge0$ for $p\le 2$ and  always $\le 0$
for $p\ge2$. This is not too surprising. It more easily follows from the fact that $I+H(\phi)$ is symmetric and that $H^p(\T)\subseteq H^q(\T)$ if $p\ge q$.

Another point is that $\rho$ seems to have two different representations. (Of course, they are the same if the $\gamma$ and $\delta$ parameters coincide.) This suggest that there ought to be another more ``symmetric'' representation. Indeed, introduce the functions
\be\label{f.zero}
v_{\tau_0,\alpha}(e^{i\theta})=(2-2\cos(\theta-\theta_0))^\alpha
\ee
with  $\tau_0=e^{i\theta_0}$, and $\Re\alpha>-1/2$.
Using 
$$
\eta_{\tau,\gamma}\delta_{\tau,\delta}=v_{\tau,(\gamma+\delta)/2} u_{\tau,(\gamma-\delta)/2}
$$
we can write $\rho=\rho_0\rho_1$
with 
\bq
\rho_0 &=& a_{+,0}\ta_{+,0} v_{1,\gamma^++\delta^+}v_{-1,\gamma^-+\delta^-+1}
\prod_{r=1}^R v_{\tau_r,(\gamma_r+\delta_r)/2}v_{\bt_r,(\gamma_r+\delta_r)/2}
\\
\rho_1 &=& t^{-m+n} u_{1,n_+}u_{-1,n_-}\prod_{r=1}^R u_{\tau_r,n_r/2}u_{\bt_r,n_r/2}
\ = \ (-1)^{n_+} \prod_{r=1}^R (\chi_{\tau_r})^{n_r}
\eq
with $n_\pm=\gamma^\pm-\delta^\pm$, $n_r=\gamma_r-\delta_r$, and $\tau_r=e^{i\theta_r}$, $\theta_r\in(0,\pi)$, and
$$
\chi_{\tau_r}(e^{ix}) =\left\{
\ba{rl} 1 & \mbox{ if } -\theta_r<x<\theta_r\\
-1 & \mbox{ if }\theta_r<x<2\pi-\theta_r.
\ea\right.
$$
In this representation of $\rho$, it is immediatley clear that both $\rho_0$ and $\rho_1$ are invariant under exchange of $\gamma$ and $\delta$
parameters. Notice that both $\rho_0$ and $\rho_1$ are even functions. 
It is remarkable that $\rho_0$ contains functions (\ref{f.zero}) with zeros or poles only, while
$\rho_1$ is a very particular jump function.

After these remarks, let us state the invertibility criteria whose proof is straightforward.

\begin{corollary}
Let $\phi\in G PC_0$, and assume that $\phi\tilde{\phi}=1$. Then 
$I+H(\tilde{\phi})$ is invertible on $\Hp$ if and only if $\phi$ can be written in the form 
$$
\phi=t^{2\kappa} \cdot a_0\cdot u_{1,\beta^+}\cdot u_{-1,\beta^-}\prod_{r=1}^R u_{\tau_r,\beta_r}\cdot u_{\bt_r,\beta_r}
$$
such that $\kappa\in\Z$, $a_0\in G_0 C(\T)$, $a_0\ta_0=1$, $a_0(\pm1)=1$, and
\begin{enumerate}
\item[(i)]
$-\min\{\frac{1}{p},\frac{1}{q}\}<\Re \beta^+<\min\{\frac{1}{p},\frac{1}{q}\}+1$,
\item[(ii)]
$-1-\min\{\frac{1}{p},\frac{1}{q}\}<\Re \beta^-<\min\{\frac{1}{p},\frac{1}{q}\}$,
\item[(iii)]
$-\min\{\frac{1}{p},\frac{1}{q}\}<\Re \beta_r<\min\{\frac{1}{p},\frac{1}{q}\}$ for each $1\le r\le R$,
\end{enumerate}
and if either $\kappa=0$ or $\kappa>0$ and $A_{\kappa,\kappa}$ is non-singular,
where $A_{\kappa,\kappa}=[\rho_{i-j}+\rho_{i+j}]_{i,j=0}^{\kappa-1}$, $a_0=a_{0,+}\cdot (\ta_{0,+})\iv$  is a factorization (\ref{a0-f}), and
$$
\rho=a_{0,+} \cdot \ta_{0,+} \cdot v_{1,\beta^+}\cdot v_{-1,\beta^-+1}\prod_{r=1}^R
v_{\tau_r,\beta_r} \cdot v_{\bt_r,\beta_r}.
$$
\end{corollary}

\begin{example}\em
Let us consider the special case of $I+H(\tilde{\phi})$ with 
$$
\phi=t^{2\kappa} u_{1,\beta^+}u_{-1,\beta^-}.
$$
First notice that this representation is not unique, but that we can add to $\beta^\pm$ integer multiples of
$2$ at the expense of modifying $\kappa\in\Z$. Due to the previous result, if $I+H(\tilde{\phi})$ is invertible, then we must be able to choose $\beta^+$ and $\beta^-$ fulfill the conditions (i) and (ii).
Notice that for $p=2$, this is a requirement that $\Re \beta^\pm$ lie in an open interval of length 2,
while for $p\neq2$, the interval is shorter.

Now assume that $\beta^\pm$ are chosen in this way, which fixes $\kappa$ uniquely.
For $\kappa=0$ we have invertibiliy, while for $\kappa<0$ we have not.
In case $\kappa>0$ we are led to $A_{\kappa,\kappa}$
with $\rho=v_{1,\beta^+} v_{-1,\beta^-+1}$. The matrix $A_{\kappa,\kappa}$
is a particular Toeplitz+Hankel matrix which arises from the finite section of $T(\rho)+H(t\rho)$.
In Theorem 6 (Case (4)) of \cite{BCE} it has been shown that 
$$
\det P_{\kappa} (T(\rho)+H(t\rho))P_{\kappa} = 4 \det 
\left(\frac{1}{\pi}\int_{-1}^1 \sigma(x)(2x)^{j+k}\, dx\right)_{j,k=0}^{\kappa-1}
$$
with 
$$
\sigma(\cos \theta)=\rho(e^{i\theta})(2+2\cos(\theta))^{-1/2}(2-2\cos(\theta))^{-1/2}.
$$
Hence
$$
\sigma(x) = (2-2x)^{\beta^+-1/2}(2+2x)^{\beta^-+1/2}
$$
noting that $\alpha:=\beta^+-1/2$, $\beta:=\beta^-+1/2$, yields $\Re\alpha>-1$ and $\Re\beta>-1$. The Hankel determinant above can be expressed in terms of the leading coefficients $\sigma_n$ of the (normalized) Jacobi polynomials,
$$
(\sigma_n)^2= \left( 2^{-n}{2n+\alpha+\beta \choose n }\right)^2
\frac{2n+\alpha+\beta+1}{2^{\alpha+\beta+1}}\frac{\Gamma(n+1)\Gamma(n+\alpha+\beta+1)}{\Gamma(n+\alpha+1)\Gamma(n+\beta+1)}.
$$
Thus the determinant evaluates to
$$
\frac{4\cdot 2^{\kappa(\kappa-1)}}{\pi^\kappa} \prod_{n=0}^{\kappa-1} (\sigma_n)^{-2}.
$$
This values is non-zero and thus $A_{\kappa,\kappa}$ is non-singular.
It follows that $I+H(\tilde{\phi})$ is invertible for $\kappa\ge0$.
\qed
\end{example}

This result generalizes results established in   \cite[Sect.~3.2]{BE5}, \cite[Sect.~4.1]{E3}, \cite[Thm.~3.2]{E4}, and \cite[Sect.~2.2]{E5}.

%%%%%%%%%%%%%%%%%

%\newpage

\end{document}